\documentclass[11pt]{article}
\usepackage{psfrag,cite,amsmath,graphicx,epsfig,latexsym,subfigure,color}
\usepackage[scriptsize,bf]{caption}
\usepackage{color}
\usepackage{amssymb}
\DeclareMathOperator{\argmin}{\mbox{argmin}}
\def\bn{\hfill \\ \smallskip\noindent}
\def\argmin{\mathop{\rm argmin}}

\def\vx{x}

\def\proj{\mbox{proj}}
\def\prox{\mbox{prox}}

\renewcommand{\baselinestretch}{1.2}   %line space adjusted here
\newcommand{\beq}{\begin{equation}}
\newcommand{\eeq}{\end{equation}}
\newcommand{\st}{{\rm s.t.}}

\newcommand{\trace}{{\mbox{\textrm{\rm Tr}}}}

\newcommand{\cC}{{\mbox{$\mathcal{C}$}}}

\begin{document}
\bigskip
\def\theequation {\thesection.\arabic{equation}}
\def\pn {\par\smallskip\noindent}
\def \bn {\hfill \\ \smallskip\noindent}
\newcommand{\fs}{f_1,\ldots,f_s}
\newcommand{\f}{\vec{f}}
\newcommand{\hx}{\hat{x}}
\newcommand{\hM}{\widehat{M}}
\newcommand{\hy}{\hat{y}}
\newcommand{\barhx}{\bar{\hat{x}}}
\newcommand{\vecx}{x_1,\ldots,x_m}
\newcommand{\xoy}{x\rightarrow y}
\newcommand{\barx}{{\bar x}}
\newcommand{\bary}{{\bar y}}
\newcommand{\tx}{\widetilde{x}}
\newcommand{\tz}{\widetilde{z}}
\newcommand{\ty}{\widetilde{y}}
\newcommand{\tX}{\widetilde{X}}
\newcommand{\tf}{\widetilde{f}}
\newcommand{\tu}{\widetilde{u}}
\newcommand{\tw}{\widetilde{w}}
\newcommand{\tM}{\widetilde{M}}
\newtheorem{theorem}{Theorem}[section]
\newtheorem{lemma}{Lemma}[section]
\newtheorem{corollary}{Corollary}[section]
\newtheorem{proposition}{Proposition}[section]
\newtheorem{definition}{Definition}[section]
\newtheorem{claim}{Claim}[section]
\newtheorem{remark}{Remark}[section]
\newtheorem{example}{Example}[section]

\newcommand{\newsection}{\setcounter{equation}{0}\section}

\def\br{\break}
\def\smskip{\par\vskip 5 pt}
\def\proof{\bn {\bf Proof.} }
\def\QED{\hfill{\bf Q.E.D.}\smskip}
\def\qed{\quad{\bf q.e.d.}\smskip}

\newcommand{\cM}{\mathcal{M}}
\newcommand{\cJ}{\mathcal{J}}
\newcommand{\cT}{\mathcal{T}}
\newcommand{\bx}{\mathbf{x}}
\newcommand{\bp}{\mathbf{p}}
\newcommand{\bz}{\mathbf{z}}

%% PUT YOUR TITLE PAGE INFORMATION HERE %%%
\title{\bf Iteration Complexity Analysis of Block Coordinate
Descent Methods}
 \vskip 0.5cm
\author{{Mingyi Hong}\thanks{Department of Industrial and
Manufacturing Systems Engineering, Iowa State University, IA, USA. Email: \texttt{mingyi@iastate.edu}}, \ {Xiangfeng Wang}\thanks{ Shanghai Key Lab. of Trustworthy Computing, Software Engineering Institute, East China Normal University, Shanghai 200062, China. Email: \texttt{xfwang@sei.ecnu.edu.cn}}, \ {Meisam Razaviyayn\thanks{Department of Electrical Engineering, Stanford University, Palo Alto, CA, USA. Email: \texttt{meisam@stanford.edu}}\ \ and Zhi-Quan Luo\thanks{Department of Electrical and
Computer Engineering, University of Minnesota, Minneapolis, MN, USA. Email: \texttt{luozq@umn.edu}}}
\\[20pt]}
\maketitle
\vskip 1.0cm

\begin{abstract}
% We consider a family of general block coordinate descent (BCD) methods in which a single block of variables is updated at each iteration. Such BCD type methods are well suited  for large scale nonsmooth constrained convex optimization involving multiple block variables since they solve an easy low dimensional problem at each iteration, often in close form.   Despite their popularity, the characterization of the iteration complexity for the BCD-type algorithms has remained incomplete to date, especially for multi-block nonsmooth problems and when the block variables are updated in a deterministic fashion, for example by using the classic Guass-Seidel (G-S), Gauss-Southwell (G-So), or Essentially Cyclic (E-C) rules.

 In this paper, we provide a unified iteration complexity analysis for a family of general block coordinate descent (BCD) methods, covering popular methods such as the block coordinate gradient descent (BCGD) and the block coordinate proximal gradient (BCPG), under various different coordinate update rules. We unify these algorithms under the so-called Block Successive Upper-bound Minimization (BSUM) framework, and show that for a broad class of multi-block nonsmooth convex problems, all algorithms covered by the BSUM framework achieve a global sublinear iteration complexity of ${\cal{O}}(1/r)$, where $r$ is the iteration index. { Moreover, for the case of block coordinate minimization (BCM) where each block is minimized exactly, we establish the sublinear convergence rate of $O(1/r)$ without per block strong convexity assumption. Further, we show that when there are only two blocks of variables, a special BSUM algorithm with Gauss-Seidel rule can be accelerated to achieve an improved rate of ${\cal{O}}(1/r^2)$.}
\end{abstract}

\vspace*{\fill}

%\noindent {\bf KEY WORDS:} Iteration complexity, rate of convergence, block coordinate descent, randomized block coordinate descent, nonsmooth optimization.
%\pn
%\noindent {\bf AMS(MOS) Subject Classifications:}  49, 90.

%%SO MUCH FOR THE TITLE PAGE

\vspace{-1.2cm}
\newsection{\bf Introduction} \label{sub:intro}
\vspace{-0.6cm}
Consider the problem of minimizing a nonsmooth convex function $f(x)$ of the form:
\begin{equation}\label{eq:bcd_problem}
\vspace{-0.1cm}
\begin{array}{ll}
\mbox{minimize} & \displaystyle f(x):=g\left(x_1,\cdots, x_K\right)+\sum_{k=1}^{K}h_k(x_k)\\ [10pt]
\mbox{subject to}  & x_k\in X_k,\quad k=1,...,K
\end{array}
\end{equation}
\vspace{-0.1cm}
where $g(\cdot)$ is a smooth convex function; $h_k$ is a nonsmooth convex function (possibly with extended values); $x=(x_1^T,...,x_K^T)^T\in\mathbb{R}^n$ is a partition of
the optimization variable $x$, with $x_k\in X_k\subseteq\mathbb{R}^{n_k}$. Let  $X:=\prod_{k=1}^{K}X_k$ denote the
feasible set for $x$.

A well known family of algorithms for solving \eqref{eq:bcd_problem} is the block coordinate descent (BCD) type method whereby, at every iteration a single block of variables is optimized  while the remaining blocks are held fixed. One of the best known algorithms in the BCD family is the block coordinate minimization (BCM) algorithm, where at iteration~$r$, the blocks are updated by solving the following problem {exactly} \cite{bertsekas99}
\begin{equation}
\label{eq:GSperblock}
\begin{split}
x_k^r \in \arg\min_{x_k\in X_k} \quad &g(x_1^{r},\ldots,x_{k-1}^{r},x_k,x_{k+1}^{r-1},\ldots, x_K^{r-1})+h_k(x_k), \ k=1,\cdots,
K.\end{split}
\end{equation}
When problem \eqref{eq:GSperblock} is not easily solvable, a popular variant is to solve an approximate version of problem \eqref{eq:GSperblock}, yielding the so-called block coordinate gradient descent (BCGD) algorithm, or the block coordinate proximal gradient (BCPG) algorithm in the presence of nonsmooth function \cite{tseng09coordiate, zhang13linear, shalev11, Beck13}. In particular, at a given iteration $r$, the following problem is solved for each block $k$:
\begin{equation}
\label{eq:BCPB}
\begin{split}
x_k^r = \arg\min_{x_k\in X_k}\; &\langle\nabla_k  g(x_1^{r},\ldots,x_{k-1}^{r},x_{k}^{r-1},\ldots, x_K^{r-1}), x_k-x^{r-1}_k\rangle+\frac{L_k}{2}\|x_k-x^{r-1}_k\|^2+h_k(x_k)\end{split}
\end{equation}
where $L_k>0$ is some appropriately chosen constant. Other variants of the BCD-type algorithm include those that solve different subproblems \cite{Razaviyayn12SUM}, or those with different block selection rules, such as the Gauss-Seidel (G-S) rule, the Gauss-Southwell (G-So) rule \cite{tseng09}, the randomized rule \cite{nestrov12}, the essentially cyclic (E-C) rule \cite{tseng01}, or the maximum block improvement (MBI) rule \cite{Chen2012MBI}.

In all the above mentioned variants of BCD method, each step involves solving a simple subproblem of small size, therefore the BCD method can be quite effective for solving large-scale problems; see e.g., \cite{Friedman10, Razaviyayn12SUM, Saha10, shalev11, nestrov12} and the references therein.
The existing analysis of the BCD method \cite{tseng01,bertsekas96,bertsekas97, ortega72} requires the uniqueness of the minimizer for each subproblem \eqref{eq:GSperblock}, or the quasi convexity of $f$ \cite{Grippo00}.
Recently, a unified BCD-type framework, termed the block successive upper-bound minimization (BSUM) method, is proposed in \cite{Razaviyayn12SUM}. At each iteration of the BSUM method, certain approximate function of the per-block subproblem \eqref{eq:GSperblock} is constructed and optimized. Due to the flexibility in choosing the approximate function, the BSUM includes many BCD-type algorithms as special cases. It is shown in \cite{Razaviyayn12SUM} that the method converges to stationary solutions for nonconvex problems and to global optimal solutions for convex problems, as long as certain regularity conditions are satisfied for the per-block subproblems.

The global rate of convergence for BCD-type algorithm has been studied extensively. When the objective function is strongly convex, the BCD algorithm converges globally linearly \cite{luo93errorbound:10.1007/BF02096261}. When the objective function is smooth and not strongly convex, Luo and Tseng have shown that the BCD method with the classic G-S/G-So update rules converges linearly, provided that a certain local error bound is satisfied around the solution set \cite{Luo92CD,Luo92linear_convergence,Luo93dual, luo93errorbound:10.1007/BF02096261}. In addition, such linear rate is global when the feasible set is compact. This line of analysis has recently been extended to allow certain class of nonsmooth functions in the objective \cite{zhang13linear, tseng09approximation}. For more general problems where the objective is not strongly convex and the error bound condition does not hold, several recent studies have established the  $\mathcal{O}(1/r)$ iteration complexity for various BCD-type algorithms including the randomized BCGD algorithm \cite{nestrov12}, and for more general settings with nonsmooth objective as well \cite{richtarik12, shalev11, lu13complexity}.
When the coordinates are updated according to the traditional G-S/G-So/E-C rule, however, the literature on the iteration complexity for the BCD-type algorithm is scarce. In \cite{Saha10}, Saha and Tewari have proven the $\mathcal{O}(1/r)$ rate for the G-S BCPG algorithm when applied to certain special $\ell_1$ minimization problem. In \cite{Beck13}, Beck and Tetruashvili have shown the $\mathcal{O}(1/r)$ sublinear convergence for the G-S BCGD algorithm for constrained smooth problems. In \cite{Beck13b}, Beck has shown the sublinear convergence for the G-S BCM algorithm (termed Alternating Minimization method therein) when the number of blocks is two. {Although the BCD-type algorithm with G-S rule sometimes has been found to perform better than its randomized counterpart (see, e.g., \cite{Saha10}), establishing its iteration complexity in a general multi-block nonsmooth setting is challenging \cite{nestrov12}}. To the best of our knowledge, the iteration complexity of the BCD-type algorithm with the classic G-S update rule has not yet been characterized for multi-block nonsmooth problems, not to mention other types of deterministic coordinate selection rules such as G-So, E-C or MBI. {Further, there has been no iteration complexity analysis for the classic BCM iteration \eqref{eq:GSperblock} when the number of variable blocks is more than two (i.e., $K\ge 3$).}

{In this paper, we provide a unified iteration complexity analysis for $K$-block BCD-type algorithm by utilizing the BSUM framework \cite{Razaviyayn12SUM}. Our result covers many different BCD-type algorithms such as BCM, BCPG, and BCGD under a number of deterministic coordinate update rules. First, for a broad class of nonsmooth convex problems, we show that the BSUM algorithm achieves a global sublinear convergence rate of ${\cal{O}}({1}/{r})$, provided that {\it each subproblem} is strongly convex. Second, when the number of variable blocks is two, we establish an improved ${\cal{O}}({1}/{r^2})$ rate for a particular version of the BSUM algorithm, without the strong convexity of the subproblems, or the gradient Lipschitz continuity of one of the subproblems.   Third, for the BCM algorithm \eqref{eq:GSperblock}, we show the global convergence rate of  ${\cal{O}}({1}/{r})$ without the per-block strong convexity assumption. The main results of this paper are summarized in the following table\footnote{We have used the following abbreviations: NS={\bf N}on{\bf s}mooth, C={\bf C}onstrained, K={\bf K}-block, BSC={\bf B}lock-wise {\bf S}trongly {\bf C}onvex, G-So={\bf G}auss-{\bf So}uthwell, G-S={\bf G}auss-{\bf S}eidel, E-C={\bf E}ssentially {\bf C}yclic, MBI={\bf M}aximum {\bf B}lock {\bf I}mprovement. The notion of {\it valid upper-bound} as well as the function $u_k$ will be introduced in Section 2. }}.
\begin{table}
\caption{\small {Summary of the Results}}\vspace*{-0.1cm}
\begin{center}
{\small
\begin{tabular}{|c |c |c |c| c| }
\hline
{\bf Method }& {\bf Update Rule}& {\bf Problem} & {\bf Assumptions} & {\bf Rate}\\
\hline
\hline
{\bf BSUM} & {\bf G-S/E-C}& NS+C+K& {$u_k$} valid upper-bound & ${\cal{O}}({1}/{r})$\\
\hline
{\bf BSUM} & {\bf G-So/MBI}& NS+C+K& {$u_k$} valid upper-bound, {$h$} Lipchitz & ${\cal{O}}({1}/{r})$\\
\hline
{{\bf BSUM}} & {\bf G-S}& NS+C+2& {$u_1$} valid upper-bound without BSC, $u_2=g$ & ${\cal{O}}({1}/{r})$\\
 \hline
 {{\bf BSUM}} & N/A & NS+C+1& {$u_1$} valid upper-bound without BSC & ${\cal{O}}({1}/{r})$\\
\hline
\hline
{\bf BCM} & {\bf MBI}& NS+C+K& {$h$} Lipchitz, without BSC & ${\cal{O}}({1}/{r})$\\
\hline
{\bf {BCM}} & {\bf G-S/E-C}& NS+C+K& $u_k=g$, without BSC & ${\cal{O}}({1}/{r})$\\
\hline
%{\bf BCM} & {\bf G-S/E-C}&NS+C+K & {$g$} composite/$L2\_SVM$/BSC & ${\cal{O}}({1}/{r})$\\
%\hline
%{\bf BCM} & {\bf Modified G-So}&NS+C+K & {$h$} Lipchitz, $u_k=g$, without BSC & ${\cal{O}}({1}/{r})$\\
% \hline
\hline
{\bf Accelerated BSUM} & {\bf G-S}& NS+C+2& {$u_1$} valid upper-bound, $u_2=g$, BSC & ${\cal{O}}({1}/{r^2})$\\
 \hline
\end{tabular} } \label{tableSymbols}
\end{center}
\vspace*{-0.1cm}
\end{table}

{\bf Notations:} For a given matrix $A$, we use $A[i,j]$ to denote its $(i,j)$th element. For a symmetric matrix $A$ use $\rho_{\max}(A)$ to denote its spectral norm. For a given vector $x$, we use $x[j]$ to denote its $j$th component; use $\|x\|$ to denote its $\ell_2$ norm.  We use $I_{X}(\cdot)$ to denote the indicator function for a given set $X$, i.e., $I_{X}(y)=1$ if $y\in X$, and $I_{X}(y)=\infty$ if $y\notin X$.  Let $x_{-k}$ denote the vector $x$ with $x_k$ removed. For a given function $f(x_1,\cdots,x_K)$ which contains $K$ block variables, we use $\nabla_k f(x_1,\cdots, X_k)$ to denote the partial gradient with respect to its $k$th block variable. We use $\partial f$ to denote a subgradient of a function $f$. For a given convex nonsmooth function $\ell(\cdot)$, we define the proximity operator $\prox_{\ell}(\cdot):\mathbb{R}^n\mapsto \mathbb{R}^n$ as
\[
\prox^{\beta}_{\ell}(\vx)={\argmin_{u\in\mathbb{R}^n}}\;\;\ell(u)+\frac\beta2\|\vx-u\|^2.
\]
Similarly, for a given convex set $X$, the projection operator  ${\rm proj}_{X}(\cdot):\mathbb{R}^n\mapsto X$ is defined as
\[
{\rm proj}_{X}(\vx)={\argmin_{u\in X}}\;\;\frac12\|\vx-u\|^2.
\]

\newsection{The BSUM Algorithm and Preliminaries}

\subsection{The BSUM Algorithm}
In this paper, we consider a family of block coordinate descent methods (BCD) for solving problem \eqref{eq:bcd_problem}. The family of the algorithms we consider falls in the general category of block successive upper-bound minimization (BSUM) method, in which certain {\it approximate version} of the objective function is optimized one block variable at a time, while fixing the rest of the block variables \cite{Razaviyayn12SUM}. In particular, at iteration $r+1$, we first pick an index set $\cC^{r+1}\subseteq\{1,\cdots,K\}$. Then the $k$th block variable is updated by
\begin{align} \label{eq:BSUM}
x^{r+1}_k
\left\{ \begin{array}{ll}\in \min_{x_k\in X_k}\; u_k\left(x_k; x^{r+1}_{1},\cdots, x^{r+1}_{k-1}, x^{r}_{k},\cdots, x^r_{K}\right)+h_k(x_k), & \mbox{if}\; k\in\cC^{r+1};\\
=x^{r}_k, &\mbox{if}\; k\notin\cC^{r+1},
\end{array}\right.
\end{align}
where $u_k(\cdot; x^{r+1}_{1},\cdots, x^{r+1}_{k-1}, x^{r}_{k},\cdots, x^r_{K})$ is an approximation of $g(x)$ at a given iterate $(x^{r+1}_{1},\cdots, x^{r+1}_{k-1}, x^{r}_{k},\cdots, x^r_{K})$.  We will see shortly that by properly specifying the approximation function $u_k(\cdot)$ as well as the index set $\cC^{r+1}$, we can recover many popular BCD-type algorithms such as the BCM, the BCGD, the BCPG methods and so on.

To simplify notations, let us define a set of auxiliary variables
\begin{align}
w^r_k&:=[x^{r}_1,\cdots, x^{r}_{k-1}, x^{r-1}_{k}, x^{r-1}_{k+1}, \cdots, x^{r-1}_K], \ k=1,\cdots, K,\nonumber\\
w^r_{-k}&:=[x^{r}_1,\cdots, x^{r}_{k-1}, x^{r-1}_{k+1}, \cdots, x^{r-1}_K], \ k=1,\cdots, K,\nonumber\\
x_{-k}&:=[x_1,\cdots, x_{k-1}, x_{k+1}, \cdots, x_K]. \nonumber
\end{align}
Clearly we have $w^{r}_{K+1}:=x^r, \; w^{r}_1:=x^{r-1}$.
Moreover, at each iteration $r+1$, define a set of new variables $\{\hx_k^{r+1}\}_{k=1}^{K}$ as follows
\begin{align}\label{eq:def_hatx}
\hx^{r+1}_k\in \min_{x_k\in X_k}\; u_k\left(x_k; x^{r}\right)+h_k(x_k),  \; k=1,\cdots, K.
\end{align}
Clearly $\{\hx^{r+1}_k\}_{k=1}^{K}$ represents a ``virtual" update where all variables are optimized in a Jacobi manner based on $x^r$.

The BSUM algorithm is described formally in the following table.
\begin{center}
\fbox{
\begin{minipage}{5.2in}
\smallskip
\centerline{\bf The Block Successive Upper-Bound Minimization (BSUM) Algorithm}
\smallskip
At each iteration $r+1$, pick an index set $\cC^{r+1}$; \\
\quad {\bf For} $k=1,\cdots, K$, do:
\begin{equation}%\label{eq:BSUM}
x^{r+1}_k
\left\{ \begin{array}{ll}\in \min_{x_k\in X_k}\; u_k\left(x_k; w^{r+1}_k\right)+h_k(x_k), & \mbox{if}\; k\in\cC^{r+1};\\
=x^{r}_k, &\mbox{if}\; k\notin\cC^{r+1}\nonumber
\end{array}\right..\\
\end{equation}
{\bf End For}.
\end{minipage}
}
\end{center}

In this paper, we consider four well-known block selection rules, described below:
\begin{enumerate}
\item {\it Gauss-Seidel (G-S) rule}: At each iteration $r+1$ all the indices are chosen, i.e., $\cC^{r+1}=\{1,\cdots,K\}$. Using this rule, the blocks are updated cyclically with fixed order.
\item {\it Essentially cyclic (E-C) rule}: There exists a given period $T\ge 1$ during which each index is updated at least once, i.e.,
\begin{align}
\bigcup_{i=1}^{T}\cC^{r+i}=\{1,\cdots,K\}, \; \forall~r.
\end{align}
We call this update rule a {\it period-$T$} essentially cyclic update rule. Clearly when $T=1$ we recover the G-S rule.
\item  {\it Gauss-Southwell (G-So) rule}: At each iteration $r+1$, $\cC^{r+1}$ contains a single index $k^*$ that satisfies:
\begin{align}\label{eq:G-So}
k^*\in \left\{k \;\bigg{|}\; \|\hx^{r+1}_{k}-x^{r}_{k}\|\ge q \max_{j}\|\hx^{r+1}_j-x^{r}_j\|\right\}
\end{align}
for some constant $q\in (0,\; 1]$.

\item {\it Maximum block improvement (MBI) rule}: At each iteration $r+1$, $\cC^{r+1}$ contains a single index $k^*$ that satisfies:
\begin{align}
k^*\in \arg\max_{k}-f(\hx^{r+1}_k, x^{r}_{-k}).
\end{align}
\end{enumerate}

\subsection{Main Assumptions} \label{sub:assumptions}
Suppose $f$ is a closed proper convex function in $\mathbb{R}^n$.
Let ${\rm dom}\ f$ denote the effective domain of $f$ and let
$\hbox{int}(\hbox{dom } f)$ denote the interior of ${\rm dom}\ f$.
We make the following standing assumptions regarding problem \eqref{eq:bcd_problem}:
\pn {\bf Assumption A.}
\begin{itemize}
\item [(a)] Problem \eqref{eq:bcd_problem} is a convex problem, ant its global minimum is attained. The intersection $X\cap \hbox{int}(\hbox{dom } f)$ is nonempty.
\item
[(b)] The gradient of $g(\cdot)$ is block-coordinate-wise uniformly Lipschitz continuous
\begin{align}
\|\nabla_k g([x_{-k}, x_k])-\nabla_k g([x_{-k},{x}'_k])\|\le M_k \|x_k-{x}'_k\|,\quad\forall~x_k,x_k'\in X_k, ~ \forall~x \in X, \ \forall\ k \label{eq:gk_lipchitz}
\end{align}
where $M_k>0$ is a constant. Define $M_{\max}=\max_k M_k$.

The gradient of $g(\cdot) $ is also uniformly Lipschitz continuous
\begin{align}
&\|\nabla g(x)-\nabla g(x')\|\le
M\|x-x'\|,~\quad\quad\forall~x,x'\in X\label{eq:g_lipchitz}
\end{align}
where $M>0$ is a constant.
\end{itemize}

Next we make the following assumptions regarding the approximation function $u_k(\cdot;\cdot)$ in \eqref{eq:BSUM}.
\pn {\bf Assumption B.}
\begin{itemize}
\item [(a)]  $u_k(x_k; x)= g(x), \quad \forall\; x\in {X}, \ \forall\; k,$
\item [(b)] $u_k(v_k; x) \geq g(v_k,x_{-k}),\quad\; \forall\; v_k \in {X}_k, \ \forall\; x \in{X}, \ \forall\; k,$
\item [(c)] $\nabla u_k(x_k;x)= \nabla_{k}g(x), \quad\; \forall\; x\in X, \; \forall\; k, $
\item [(d)] $u_k(v_k; x)$ is continuous in $v_k$ and $x$. Further, for any given $x$,  it is strongly convex in $v_k$
$$u_k(v_k; x)\ge u_k(\hat{v}_k; x)+\langle\nabla u_k(\hat{v}_k; x),v_k-\hat{v}_k\rangle
+\frac{\gamma_k}{2}\|v_k-\hat{v}_k\|^2,\ \forall~v_k, \ \hat{v}_k\in X_k, \ \forall~x\in X$$
where $\gamma_k>0$ is independent of the choice of $x$.
\item [(e)] For any given $x$, $u_k(v_k; x)$ has Lipschitz continuous gradient, that is
\begin{align}\label{eq:uk_lipchitz}
\|\nabla u_k(v_k; x)-\nabla u_k(\hat{v}_k; x)\|\le L_k\|v_k-\hat{v}_k\|,\ \forall\ \hat{v}_k,\ v_k\in X_k, \ \forall
\ k, \ \forall~x\in X,
\end{align}
where $L_k>0$ is some constant.  Further, we have
\begin{align}\label{eq:uk_lipchitz2}
\|\nabla u_k(v_k; x)-\nabla u_k({v}_k; y)\|\le G_k\|x-y\|,\ \forall\ v_k\in X_k, \ \forall
\ k, \ \forall~x, y\in X.
\end{align}
Define $L_{\rm max}:=\max_{k}L_k$; $G_{\rm max}:=\max_{k}G_k$.

\end{itemize}
We refer to the $u_k$'s that satisfy Assumption B as a {\it valid upper-bound}.

A few remarks are in order regarding to the assumptions made above.

First of all, Assumption B indicates that for any given $x$, each $u_k(\cdot; x)$ is a locally tight upper bound
for $g(x)$. When the approximation function is chosen as the original function $g(x)$, then we recover the classic BCM algorithm; cf.~\eqref{eq:GSperblock}. In many practical applications especially nonsmooth problems,
minimizing the approximation functions often leads to much simpler subproblems than directly minimizing
the original function; see e.g., \cite{HeLiaoHanYang2002,WangYuan2012, zhang11primaldual, Yang10ADMM, hong15busmm_spm}. For example, if $h_k(\cdot)=0$ for all $k$, and $u_k$ takes the following form
\begin{align}\label{eq:bcgd}
u_k(x_k; w_k^{r+1})=g(w_k^{r+1})+\langle \nabla_k g(w_k^{r+1}), x_k-x^r_k \rangle+\frac{M_k}{2}\|x_k-x^r_k\|^2,
\end{align}
then we recover the well known BCGD method \cite{Beck13,luo93errorbound:10.1007/BF02096261, nestrov12}, in which $x_k$ is updated by
\begin{align}
x^{r+1}_k=\proj_{X_k}\left[x^{r}_k-\frac{1}{M_k}\nabla_k g(w_k^{r+1})\right].
\end{align}
When the nonsmooth components $h_k$'s are present, the above choice of $u_k(\cdot;\cdot)$ in \eqref{eq:bcgd} leads to the so-called BCPG method \cite{Razaviyayn12SUM, Combettes09, zhang13linear}, in which $x_k$ is updated by \begin{align}
x^{r+1}_k=\prox^{M_k}_{h_k+I_{X_k}}\left[x^{r}_k-\frac{1}{M_k}\nabla_k g(w_k^{r+1})\right].
\end{align}
For other possible choices of the approximation function, we refer the readers to \cite{Razaviyayn12SUM, Mairal13}.

Secondly, the strong convexity requirement on $u_k(\cdot;x)$ in Assumption B(d) is quite mild,
see the examples given in the previous remark (e.g., BCPG and BCGD). When $u_k$ is chosen as the original function $g(x)$, this requirement says that $g(x)$ must be {\it block-wise strongly convex} (BSC). The BSC condition is in fact satisfied in many practical engineering problems. The following are two interesting examples.

{\begin{example}\label{ex:wireless}
Consider the rate maximization problem in an uplink wireless communication network, where $K$ users transmit to a single base station (BS) in the network. Suppose each user has $n_t$ transmit antennas, and the BS has $n_r$ receive antennas. Let $C_k\in \mathbb{R}^{n_t\times n_t}$ denote user $k$'s transmit covariance matrix, $P_k$ denote the maximum transmit power for user $k$, and $H_k\in\mathbb{R}^{n_r\times n_t}$ denote the channel matrix between user $k$ and the BS. Then the uplink channel capacity optimization problem is given by the following convex program \cite{cover05, yu04}
 \begin{align}
 \min_{\{C_k\}_{k=1}^{K}}\; -\log\det\left|\sum_{k=1}^{K}{H_k  C_k H^T_k} +I_{n_r}\right|, \quad \st\quad C_k\succeq 0,\;  \trace{[C_k]}\le P_k, \ k=1,\cdots, K, \label{eq:MAC}
 \end{align}
 where $I_{n_r}$ is the $n_r\times n_r$ identity matrix.  The celebrated iterative water-filling algorithm (IWFA) \cite{yu04} for solving this problem is simply the BSUM algorithm with exact block minimization (i.e. the BCM algorithm) and G-S update rule. It is easy to verify that when $n_t\le n_r$ (i.e., the number of transmit antenna is smaller than that of the receive antenna), and when the channels are generated randomly, then with probability one $H^T_k H_k$ is of full rank, implying that the BSC condition is satisfied. We note that there has been no iteration complexity analysis of the IWFA algorithm for any type of block selection rules.
\end{example}}

%\begin{example}\label{ex:wireless}
%Consider the rate maximization problem in an uplink wireless communication network, where $K$ users transmit to a single base station (BS) in the network. Let $x_k\in \mathbb{R}$ denote user $k$'s transmit power; $P_k$ denote the maximum transmit power for user $k$; $h_k\in\mathbb{C}$ denote the channel between user $k$ and the BS. Then the uplink channel capacity optimization problem is given by the following convex program \cite{cover05}
% \begin{align*}
% \max\; \log\left(\sum_{k=1}^{K}{|h_k|^2 x_k} +1\right), \quad \st\quad 0\le p_k\le P_k, \ k=1,\cdots, K.
% \end{align*}
% Clearly this problem is BSC, as the magnitude of the channel $|h_k|$ is always bounded away from zero. This problem can be generalized to the case where the users and the BS all equipped with multiple antennas; see \cite{yu04} for detailed discussion.
%\end{example}
\begin{example}
Consider the following LASSO problem:
 \begin{align*}
    \min\|Ax-b\|^2+\lambda \|x\|_1,
    \end{align*}
 where $A\in\mathbb{R}^{M\times K}, \; b\in\mathbb{R}^{M}$, and $x=[x_1,\cdots, x_K]^T$, with $x_k\in\mathbb{R}$ for all $k$. That is, each block consists of a single scalar variable.  In this case, as long as none of $A$'s columns are zero (in which case we simply remove that column and the corresponding block variable), the problem satisfies the BSC property. Prior to our work, there is no iteration complexity analysis for applying  BCD with deterministic block selection rules such as G-S and E-C for LASSO (with general data matrix $A$).
\end{example}

Note that the BSC property, or more generally the strong convexity assumption on the approximate function $u_k$, is reasonable as it ensures that each step of the BSUM algorithm is well-defined and has a unique solution. {In the ensuing analysis of the BSUM algorithm, we  assume that either the BSC property holds true, or $u_k$ is a valid upper-bound. Later in Sections 4 - 6, we will consider the case where the BSC assumption is absent. }

\newsection{Convergence Analysis for BSUM} \label{sec:convergence}
In this section, we show that under assumptions A and B, the BSUM algorithm with flexible update rules achieves global sublinear rate of convergence. %As the G-S rule is a special case of the period-$T$ E-C rule (with $T=1$), we will mainly focus on the rest of the update rules, i.e., the E-C, MBI and G-So rules.

Let us define $X^*$ as the optimal solution set, and let ${x}^*\in X^*$ be one of the optimal solutions. For the BSUM algorithm, define the optimality gap as
\begin{align}
\Delta^r:&=f(x^{r})-f({x}^*).
\end{align}

Despite the generality of the BSUM algorithm, our analysis of BSUM only consists of three simple steps: S1) estimate the amount of successive decrease of the optimality gaps; S2) estimate the cost yet to be minimized after each iteration; S3) estimate the rate of convergence.

We first characterize the successive difference of the optimality gaps before and after one iteration of the BSUM algorithm, with different update rules.
\begin{lemma}\label{lm:p-descent}{\rm \bf (Sufficient Descent)}
Suppose Assumption A and Assumption B hold. Then
\begin{enumerate}
\item For BSUM with either G-S rule or the E-C rule, we have that for all $r\ge 1$
\begin{equation}\label{eq:descent-GS}
\Delta^{r}-\Delta^{r+1}\ge \sum_{k=1}^{K}\frac{\gamma_k}{2}\|x_k^{r} - x_k^{r+1}\|^2\ge {\gamma}\|x^{r} - x^{r+1}\|^2,
\end{equation}
where the constant $\gamma:=\frac{1}{2}\min_k\gamma_k>0$.
\item For BSUM with G-So rule and MBI rule, we have that for all $r\ge 1$
\begin{equation}\label{eq:descent-GSo}
\Delta^{r}-\Delta^{r+1}\ge \frac{c_1}{K} {\gamma}\|x^r - \hx^{r+1}\|^2,
\end{equation}
where the constant $\gamma:=\frac{1}{2}\min_k\gamma_k>0$; For G-So rule, $c_1=q$, and for MBI rule, $c_1=1$.
\end{enumerate}
%\item For R-BCD, we have for all $t\ge 1$
%\begin{align}\label{eq:descent-rbcd}
%&\mathbb{E}[\Delta^t- \Delta^{t+1}\mid x^t] =\mathbb{E}[f(x^t)- f(x^{t+1})\mid x^t]\nonumber\\
%&\ge\sum_{k=1}^{K}\frac{\gamma_k}{2} p_k\|\vx_k^t - \hx_k^{t+1}\|^2\ge \hat{\gamma}\|\vx^t - \hx^{t+1}\|^2
%\end{align}
%where the expectation is taken over the random selection of variable block to be updated at step $t$; the constant $\hat{\gamma}:=\frac{1}{2}\min_{k}p_k\gamma_k>0$.
%\end{enumerate}
\end{lemma}
\proof We first show part (1) of the proof. Suppose that $k\notin\cC^{r+1}$, then we have the following trivial inequality
\begin{align}
f(w^{r+1}_k)- f(w^{r+1}_{k+1})\ge \frac{\gamma_k}{2}\|\vx^{r+1}_k-\vx^r_k\|^2
\end{align}
as both sides of the inequality are zero.

Suppose $k\in\cC^{r+1}$. Then using Assumption B, we have that{
\begin{align}\label{eq:sufficient_descent}
f(w^{r+1}_k)- f(w^{r+1}_{k+1})&\ge u_k(\vx^{r}_{k} ; w^{r+1}_k)+h_k(x^r_k)-\left(u_k(\vx^{r+1}_{k} ; w^{r+1}_k)+h_k(x^{r+1}_k)\right)\nonumber\\
&\ge \langle \nabla u_k(\vx^{r+1}_{k} ; w^{r+1}_k), x^r_k-x^{r+1}_k\rangle+h_k(x^r_k)-h_k(x^{r+1}_k)+\frac{\gamma_k}{2}\|\vx^{r+1}_k-\vx^r_k\|^2\nonumber\\
&\ge \langle \nabla u_k(\vx^{r+1}_{k} ; w^{r+1}_k)+\zeta_k^{r+1}, x^r_k-x^{r+1}_k\rangle+\frac{\gamma_k}{2}\|\vx^{r+1}_k-\vx^r_k\|^2\nonumber\\
&\ge \frac{\gamma_k}{2}\|\vx^{r+1}_k-\vx^r_k\|^2
\end{align}
where the first inequality is due to Assumption B(a)--B(b); the second inequality is due to Assumption B(d); in the third inequality we have defined $\zeta_k^{r+1}\in\partial h_k(\vx^{r+1}_{k})$; the last inequality is due to the fact that $x_k^{r+1}$ is the optimal solution for the strongly convex problem
$${\rm arg}\!\min_{x_k\in X_k}u_k(x_k; w^{r+1}_k)+h_k(x_k).$$}
Summing over $k$, we have
\begin{align}
f(\vx^r)- f(\vx^{r+1})\ge \gamma\|x^r - x^{r+1}\|^2,
\end{align}
where $\gamma:=\frac{1}{2}\min_{k}\gamma_k$. %Summing over $t\in(0,T]$, we obtain the desired result.

We then show part (2) of the claim. Suppose $k\in\cC^{r+1}$, then we have the following series of inequalities for the G-So rule
\begin{align}
f(x^r)- f(x^{r+1})
&=f(x^r)-f(x^r_{-k},\hat{x}^{r+1}_k)\nonumber\\
&\ge u_k(x^r_k; x^r)+h_k(x^r_k)-u_k(\hx^{r+1}_k; x^r)-h_k(\hat{x}^{r+1}_k)\nonumber\\
&\ge \frac{1}{2}\gamma_k\|x^r_k-\hat{x}^{r+1}_k\|^2\nonumber\\
&\ge \frac{q\min_{j}\gamma_j}{2K}\sum_{j=1}^{K}\|x^r_j-\hat{x}^{r+1}_j\|^2\nonumber\\
&= \frac{q}{K}{\gamma}\|\vx^r - \hx^{r+1}\|^2.
\end{align}
Similar steps lead to the result for the MBI rule.
%Next we show part (2) of the claim. We have the following
%\begin{align}\label{eq:older}
%\mathbb{E}[f(x^t)- f(x^{t+1})\mid x^t]
%&=\sum_{k=1}^{K} p_k\left[f(x^t)-f(x^t_{-k},\hat{x}^{t+1}_k)\right]\nonumber\\
%&\ge \sum_{k=1}^{K} p_k\left[u_k(x^t_k; x^t)+h_k(x^t_k)-u_k(\hx^{t+1}_k; x^t)-h_k(\hat{x}^{t+1}_k)\right]\nonumber\\
%&\ge \frac{1}{2}\sum_{k=1}^{K}p_k\gamma_k\|x^t_k-\hat{x}^{t+1}_k\|^2\nonumber\\
%&\ge \hat{\gamma}\|\vx^t - \hx^{t+1}\|^2,
%\end{align}
%where the second to the last inequality can be argued similarly as in \eqref{eq:sufficient_descent}; $\hat{\gamma}:=\frac{1}{2}\min_k p_k \lambda_k$. This completes the proof.
\QED

Next we show the second step of the proof, which estimates the gap yet to be minimized after each iteration of the algorithm. Let us define the following constants:
\begin{align}\label{eq:Ra}
\begin{split}
&R:=\max_{x\in X}\max_{x^*\in X^*}\left\{\|x-x^*\|\ : \ f(x)\le f(x^1) \right\}, \quad
Q:=\max_{x\in X}\left\{\|\nabla g(x)\|:\ f(x)\le f(x^1)\right\}.
%&R_A:=\max_{x\in X}\max_{x^*\in X^*}\left\{\sqrt{\sum_{k=1}^{K}A[k,k]\|x_k-x_k^*\|^2}\ : \ f(x)\le f(x^1) \right\},\ \mbox{for a given matrix}~A\succ 0, \\
%&R_k:=\max_{x\in X}\max_{x^*\in X^*}\left\{\|x_k-x_k^*\|\ : \ f(x)\le f(x^1) \right\}. \quad
\end{split}
\end{align}
{When assuming that the level set $\{x: f(x)\le f(x^1)\}$ is compact, then all the above constants are finite.} Clearly we have
\begin{align}
\|x^r-x^*\|\le R,\quad \|\nabla g(x^r)\|\le Q, \ \forall\ r=1,\cdots.
\end{align}

Occasionally we need to further make the assumption that the nonsmooth part $h(x)$ is Lipchitz continuous:
\begin{align}
\|h(x)-h(y)\|\le L_h\|x-y\|,\ \forall\ x, y\in X\label{eq:h_lip},
\end{align}
with some $L_h>0$. Note that such assumption is satisfied by most of the popular nonsmooth regularizers such as the $\ell_1$ norm, the $\ell_2$ norm and so on. Also note that even with this assumption, our considered problem is still a {\it constrained} one, as the convex constraints $x_k\in X_k$ have not been moved to the objective as nonsmooth indicator functions.

\begin{lemma}\label{lemma:clost_to_go} {\rm \bf (Cost-to-go Estimate)}
Suppose Assumptions A and B are satisfied. Then
\begin{enumerate}
\item For the BSUM with G-S update rule, we have
\begin{align*}
(\Delta^{r+1})^2\le R^2 K G^2_{\max}\|x^{r+1}-x^{r}\|^2,\ \forall\ x^*\in X^*.
\end{align*}
\item For the BSUM with period-$T$ E-C update rule, we have
\begin{align*}
(\Delta^{r+T})^2\le T R^2 K G^2_{\max}\sum_{t=1}^{T}\|x^{r+t}-x^{r+t-1}\|^2,\ \forall\ x^*\in X^*.
\end{align*}
\item For the BSUM with G-So and MBI rules, further assume that $h(\cdot)$ is Lipchitz continuous (cf. \eqref{eq:h_lip}). Then we have
\begin{align*}
\Delta^{r}&=f(x^{r})-f(x^*)\le 2\left((Q+L_h)^2+L^2_{\max} {K} R^2\right)\|\hx^{r+1}-x^{r}\|^2,\ \forall\ x^*\in X^*.
\end{align*}
\end{enumerate}
\end{lemma}
\begin{proof}
We first show part (1) of the claim. We have the following sequence of inequalities
\begin{align}
f(x^{r+1})-f(x^*)
&=g(x^{r+1})-g(x^*)+h(x^{r+1})-h(x^*)\nonumber\\
&\le \langle\nabla g(x^{r+1}), x^{r+1}-x^*\rangle+h(x^{r+1})-h(x^*)\nonumber\\
&=\sum_{k=1}^{K}\langle \nabla_k g(x^{r+1})-\nabla u_k(x_k^{r+1};w^{r+1}_k), x^{r+1}_k-x^{*}_k\rangle\nonumber\\
&\quad\quad+\sum_{k=1}^{K}\langle \nabla u_k(x_k^{r+1};w^{r+1}_k), x^{r+1}_k-x^{*}_k\rangle+h(x^{r+1})-h(x^*)\label{eq:cost_to_go}.
\end{align}
Notice that $x^{r+1}_k$ is the optimal solution for problem: $\argmin_{x_k\in X_k}u_k(x_k; w^{r+1}_k)+h_k(x_k)$. It follows from the optimality condition of this problem that there exists some $\zeta_k^{r+1}\in \partial \left(h_k(x^{r+1}_k)\right)$ such that
\begin{align}
0&\ge\langle \nabla u_k(x^{r+1}_k; w^{r+1}_k)+ \zeta_k^{r+1}, x^{r+1}_k-x^*_k\rangle\nonumber\\
&\ge\langle\nabla u_k(x^{r+1}_k; w^{r+1}_k), x^{r+1}_k-x^*_k\rangle+h_k(x^{r+1}_k)-h_k(x^*_k),\label{eq:negative}
\end{align}
where in the last inequality we have used the definition of subgradient
\begin{align}
h_k(x_k)-h_k(v_k)\ge \langle \zeta_k^{r+1}, x_k-v_k\rangle,\;\forall \ x_k, v_k\in X_k.
\end{align}
Combining \eqref{eq:cost_to_go} and \eqref{eq:negative}, we obtain
\begin{align}
&\left(f(x^{r+1})-f(x^*)\right)^2\nonumber\\
%&\le \left(\sum_{k=1}^{K}\langle \nabla_k g(x^{r+1})-\nabla u_k(x^{r+1}_k; w^{r+1}_k), x^{r+1}_k-x^{*}_k\rangle\right)^2\nonumber\\
&\stackrel{\rm (i)} \le \left(\sum_{k=1}^{K}\|\nabla_k g(x^{r+1})-\nabla u_k(x^{r+1}_k; w^{r+1}_k)\| \|x^{r+1}_k-x^{*}_k\|\right)^2\nonumber\\
%&\le  \left(\sum_{k=1}^{K}\left(\|\nabla_k g(x^{r+1})-\nabla u_k(x^{r}_k; w^{r+1}_k)\|+\|\nabla u_k(x^{r}_k; w^{r+1}_k)-\nabla u_k(x^{r+1}_k; w^{r+1}_k)\| \right) \|x^{r+1}_k-x^{*}_k\|\right)^2\nonumber\\
%&\stackrel{\rm (i)}\le  R^2 \sum_{k=1}^{K}\left(\|\nabla_k g(x^{r+1})-\nabla u_k(x^{r+1}_k; w^{r+1}_k)\|^2\right)\nonumber\\
&\stackrel{\rm (ii)}\le  \left(\sum_{k=1}^{K}G_k\|x^{r+1}-w^{r+1}_k\|\|x^{r+1}_k-x^{*}_k\|\right)^2\nonumber\\
%&\le  K G_{\max}^2\|x^{r+1}-x^r\|^2\sum_{k=1}^{K} \|x^{r+1}_k-x^{*}_k\|^2\nonumber\\
& \le   R^2 K G^2_{\max}\|x^{r+1}-x^r\|^2\nonumber
\end{align}
where in $\mbox{(i)}$ we have used the Cauchy-Schwarz inequality and the Lipchitz continuity of $u_k(\cdot;\cdot)$ in \eqref{eq:uk_lipchitz}; in $\mbox{(ii)}$ we have used the Lipchitz continuity of $\nabla g(\cdot)$ in \eqref{eq:g_lipchitz},  and that $\nabla_k g(x^{r+1})=\nabla_k u_k(x_k^{r+1}; x^{r+1})$ (cf. Assumption B(c)).

Next we show part (2) of the claim. Let us define a new index set $\{r_k\}$ as follows:
\begin{align}\label{eq:r_k}
r_k:=\arg\max_t \{ x_k^{t}\ne x_k^{r+T} \}+1, \; k=1,\cdots, K.
\end{align}
That is, $r_k$ is the latest iteration index (up until $r+T$) in which the $k$th variable has been updated. From this definition we have $x^{r_k}_k=x^{r+T}_k$, for all $k$.   %Also let $v^{r+T}:=[x_1^{r_1},\cdots, x^{r_k}_k]$.

%First we note that for any $k$ and $r_k$, we have
%\begin{align}
%&\sum_{k=1}^{K}\|\nabla_k g(x^{r})-\nabla u_k(x^{r_k}_k; w^{r_k}_{k})\|\nonumber\\
%&\le \sum_{k=1}^{K}\|\nabla_k g(x^{r})-\nabla u_k(x^{r_k-1}_k; w^{r_k}_{k})\|+\|u_k(x^{r_k-1}_k; w^{r_k}_{k})-\nabla u_k(x^{r_k}_k; w^{r_k}_{k})\|\nonumber\\
%&\le \sum_{k=1}^{K}M\|x^r-[x^{r_k-1}_k; w^{r_k}_{k}]\|+L_k\|x_k^{r_k}-x_k^{r_k-1}\|\nonumber\\
%&\le MK\sum_{t=1}^{T}\|x^{r+i}-x^{r+i-1}\|+\sum_{k=1}^{K}L_k\|x_k^{r_k}-x_k^{r_k-1}\|
%\end{align}
%Moreover, we have
%\begin{align}
%&\sum_{k=1}^{K}\|\nabla u_k(x^{r_k}_k; w^{r_k}_{k})\|\nonumber\\
%&\le \sum_{k=1}^{K}\|\nabla u_k(x^{r_k}_k; w^{r_k}_{k})-\nabla u_k(x^{r_k-1}_k; w^{r_k}_{k})\|+\|\nabla u_k(x^{r_k-1}_k; w^{r_k}_{k})\|\nonumber\\
%&\le \sum_{k=1}^{K}L_k\|x^{r_k}_k-x^{r_k-1}_k\|+\|\nabla g(w^{r_k}_k)\|
%\end{align}
We have the following sequence of inequalities
\begin{align*}
&f(x^{r+T})-f(x^*)\\
&=g(x^{r+T})-g(x^*)+\sum_{k=1}^{K}\left(h_k(x^{r_k}_k)-h_k(x^*_k)\right)\nonumber\\
&\le \langle \nabla g(x^{r+T}), x^{r+T}-x^*\rangle +\sum_{k=1}^{K}\left(h_k(x^{r_k}_k)-h_k(x^*_k)\right)\\
&\stackrel{\rm (i)}= \sum_{k=1}^{K}\bigg(\langle \nabla_k g(x^{r+T})-\nabla u_k(x^{r_k}_k; w^{r_k}_{k}), x_k^{r+T}-x_k^*\rangle + \langle\nabla u_k(x^{r_k}_k; w^{r_k}_{k}), x_k^{r_k}-x_k^{*}\rangle\bigg)\\
&\quad\quad+\sum_{k=1}^{K}\left(h_k(x^{r_k}_k)-h_k(x^*_k)\right)\\
&\stackrel{\rm (ii)}\le \sum_{k=1}^{K}\langle \nabla_k g(x^{r+T})-\nabla u_k(x^{r_k}_k; w^{r_k}_{k}), x_k^{r+T}-x_k^*\rangle %+\langle \nabla u_k(x^{r_k-1}_k; w^{r_k}_{k}), x_k^{r+T}-x_k^{r_k}\rangle\nonumber\\
%&\quad\quad+\langle \nabla u_k(x^{r_k-1}_k; w^{r_k}_{k})-u_k(x^{r_k}_k; w^{r_k}_{k}), x_k^{r_k}-x_k^{*}\rangle\bigg)
\end{align*}
where in $\rm (i)$ we have used the fact that $x^{r+T}_k=x_k^{r_k}$, for all k; in $\rm (ii)$ we have used the optimality of $x^{r_k}_k$.
Taking the square on both sides, we obtain
\begin{align*}
&(f(x^{r+T})-f(x^*))^2\\
&\le \left(\sum_{k=1}^{K}\|\nabla_k g(x^{r+T})-\nabla u_k(x^{r_k}_k; w^{r_k}_{k})\|\|x_k^{r+T}-x_k^*\|\right)^2\\
&\le \left(\sum_{k=1}^{K}G_k \|x^{r+T}-w^{r_k}_{k}\| \|x_k^{r+T}-x_k^*\|\right)^2\\
&\le \left(\sum_{k=1}^{K}G_k \left(\|x^{r+T}-x^{r_k}\|+\|x^{r_k}-w^{r_k}_{k}\|\right) \|x_k^{r+T}-x_k^*\|\right)^2\\
%&\le G^2_{\max} 2K \sum_{k=1}^{K}\left(\|x^{r+T}-x^{r_k}\|^2+\|x^{r_k}-w^{r_k}_{k}\|^2\right)\|x_k^{r+T}-x_k^*\|^2\\
&\le T K  G^2_{\max} R^2 \sum_{t=1}^{T}\|x^{r+t-1}-x^{r+t}\|^2.
%&(f(x^{r+T})-f(x^*))^2\\
%&\le 3 R^2\left(\sum_{k=1}^{K}M\|x^{r+T}-w^{r_k}_{k}\|\right)^2 +3 Q^2\left(\sum_{k=1}^{K}\|x_k^{r_k}-x_k^{r+T}\|\right)^2+3R^2\left(\sum_{k=1}^{K}L_k\|x^{r_k-1}_k-x^{r_k}_k\|\right)^2\\
%&\le 3 R^2 K^2 M^2 T\sum_{t=1}^{T}\|x^{r+t}-x^{r+t-1}\|^2+3 Q^2 K T \sum_{t=1}^{T}\|x^{r+t-1}-x^{r+t}\|^2+3 R^2KTL^2_{\rm max}\sum_{t=1}^{T}\|x^{r+t-1}-x^{r+t}\|^2\\
%&\le 3 K T\left( R^2 K M^2+  R^2L^2_{\rm max}+Q^2  \right) \sum_{t=1}^{T}\|x^{r+t}-x^{r+t-1}\|^2.
\end{align*}
%where in the last inequality we have again used the fact that $x^{r+T}_k=x_k^{r_k}$, for all k.

Finally we  show part (3) of the claim.  We have the following sequence of inequalities
\begin{align}
&f(x^{r})-f(x^*)\nonumber\\
&=g(x^r)-g(x^*)+h(x^r)-h(x^*)\nonumber\\
&\stackrel{\rm (i)}\le \langle \nabla g(x^r), x^r-x^*\rangle+L_h\|x^r-\hx^{r+1}\|+h(\hx^{r+1})-h(x^*)\nonumber\\
&=\langle \nabla g(x^r), x^r-\hx^{r+1}\rangle+\langle \nabla g(x^r), \hx^{r+1}-x^*\rangle+L_h\|x^r-\hx^{r+1}\|+h(\hx^{r+1})-h(x^*)\nonumber\\
&\le (L_h+Q)\|x^r-\hx^{r+1}\|+\sum_{k=1}^{K}\langle \nabla_k g(x^r)-\nabla u_k(\hx^{r+1}_k;x^r), \hx_k^{r+1}-x_k^*\rangle\nonumber\\
&\quad+\sum_{k=1}^{K}\; \left\langle \nabla u_k(\hx^{r+1}_k;x^t), \hx_k^{r+1}-x_k^*\right\rangle+h(\hx^{r+1})-h(x^*)\label{eq:r_bcd_intermediate}
\end{align}
where step $\rm (i)$ follows from the Lipchitz continuity assumption \eqref{eq:h_lip} as well as the convexity of $g(\cdot)$.
Similar to the proof of \eqref{eq:negative} in part (1), we can show that
\begin{align}
\sum_{k=1}^{K}\;\langle \nabla u_k(\hx^{r+1}_k;x^r), \hx_k^{r+1}-x_k^*\rangle+h(\hx^{r+1})-h(x^*)\le 0.
\end{align}
Moreover, it follows from Assumption B(c) and B(e) that
\begin{align}
&\left(\sum_{k=1}^{K}\;\left\langle \nabla_k g(x^r)-\nabla u_k(\hx^{r+1}_k;x^r), x_k^{r+1}-x_k^*\right\rangle\right)^2\nonumber\\
&=\left(\sum_{k=1}^{K}\;\left\langle \nabla u_k(x^r_k; x^r)-\nabla u_k(\hx^{r+1}_k;x^r), x_k^{r+1}-x_k^*\right\rangle\right)^2\nonumber\\
&\le K\sum_{k=1}^{K}L^2_k\|x^r_k-\hx^{r+1}_k\|^2\|x_k^{r+1}-x_k^*\|^2\nonumber\\
&\le K L^2_{\max}\|x^r-\hx^{r+1}\|^2R^2.\label{eq:key_complexity_rbcd}
\end{align}
Putting the above three inequalities together, we have
\begin{align}
f(x^{r})-f(x^*)\le 2\left((Q+L_h)^2+KL^2_{\max}R^2\right)\|x^r-\hx^{r+1}\|^2.
\end{align}
This completes the proof. \QED
\end{proof}

We are now ready to prove the $\mathcal{O}\left({1}/{r}\right)$  iteration complexity for the BSUM algorithm when applied to problem \eqref{eq:bcd_problem}. Our results below are more general than the recent analysis on the iteration complexity for BCD-type algorithms. The generality of our results can be seen from several fronts: 1) The family of algorithms we analyze is broad; it includes the classic BCD, the BCGD method, the BCPG methods as well as their variants based on different coordinate selection rules as special cases, while the existing works only focus on one particular algorithm; 2) When the coordinates are updated in a G-S fashion, our result covers the general multi-block nonsmooth case, where $h_k(x)$ can take any proper closed convex nonsmooth function, while existing works only cover some special cases \cite{Beck13, Saha10, Beck13b}; 3) When the coordinates are updated using other update rules such as G-So, MBI, E-C fashion, our convergence results appear to be new.

\begin{theorem}\label{thm:complexity}
Suppose Assumption A(a) and Assumption B hold true. We have the following.
\begin{enumerate}
\item Let $\{x^r\}$ be the sequence generated by the BSUM algorithm with G-S rule. Then we have
\begin{align}
\Delta^r=f(x^r)-f^*\le \frac{c_1}{\sigma_1}\frac{1}{r}, \; \forall~r\ge 1,
\end{align}
where the constants are given below
\begin{align}
{c}_1&=\max\{4\sigma_1-2, f(x^1)-f^*,2\}\nonumber,\\
\sigma_1&=\frac{\gamma}{KG^2_{\max}R^2},\quad
\end{align}
\item Let $\{x^r\}$ be the sequence generated by the BSUM algorithm with E-C rule. Then we have
\begin{align}
\Delta^r=f(x^r)-f^*\le \frac{c_2}{\sigma_2}\frac{1}{r-T}, \; \forall~r>T,
\end{align}
where the constants are given below
\begin{align}
{c}_2&=\max\{4\sigma_2-2, f(x^1)-f^*,2\}\nonumber,\\
\sigma_2&=\frac{\gamma}{ K {T} R^2 G^2_{\rm max}}.
\end{align}
\item Suppose the Lipchitz continuity assumption \eqref{eq:h_lip} holds true.  Let $\{\bx^r\}$ be the sequence generated by the BSUM algorithm with G-So and MBI rule. Then we have
\begin{align}
\Delta^r=f(x^r)-f^*\le \frac{1}{{\sigma}_3 r}
\end{align}
where
\begin{align}
{\sigma}_3=\left\{\begin{array}{ll}
\frac{{\gamma}q}{2K\left((Q+L_h)^2+L^2_{\rm max}KR^2\right)},\quad & \mbox{\rm (G-So rule)}\\
\frac{{\gamma}}{2K\left((Q+L_h)^2+L^2_{\rm max}KR^2\right)},\quad & \mbox{\rm (MBI rule)}
\end{array}\right..
\end{align}
\end{enumerate}

\end{theorem}
\begin{proof}
We first show part (1) of the claim by mathematical induction on $r$. From Lemma \ref{lemma:clost_to_go} and Lemma \ref{lm:p-descent}, we have that for the G-S rule, we have
\begin{align}\label{eq:key_recursion}
\Delta^r-\Delta^{r+1}\ge \frac{\gamma}{K G^2_{\max}R^2}(\Delta^{r+1})^2:=\sigma_1 (\Delta^{r+1})^2, \; \forall~r\ge 1,
\end{align}
or equivalently
\begin{align}\label{eq:key_recursion2}
\sigma_1 (\Delta^{r+1})^2+\Delta^{r+1}\le \Delta^r, \ \forall\ r\ge 1.
\end{align}
By definition, we have $\Delta^1=f(x^1)-f^*$. We first argue that
\begin{align}
\Delta^2\le \frac{{c}_1}{2\sigma_1},\ \mbox{with}\; {c}_1:=\max\{4\sigma_1-2, f(\bx^1)-f^*,2\}.
\end{align}
From \eqref{eq:key_recursion2} and the fact that $\Delta^1\le {c}_1$, we have
\begin{align}
\Delta^2&\le \frac{-1+\sqrt{1+4\sigma_1 {c}_1}}{2\sigma_1}=\frac{2 {c}_1}{1+\sqrt{1+4\sigma_1 {c}_1}}\le \frac{2 {c}_1}{1+|4\sigma_1-1|}\nonumber
\end{align}
where in the last inequality we have used the fact that ${c}_1\ge 4\sigma_1-2$. Suppose $4\sigma_1-1\ge 0$, then we immediately have $\Delta^2\le \frac{{c}_1}{2\sigma_1}$. Suppose $4\sigma_1-1<0$, then
\begin{align}
\Delta^2\le \frac{2{c}_1}{2-4\sigma_1}\le \frac{2{c}_1}{8\sigma_1-4\sigma_1}=\frac{{c}_1}{2\sigma_1}.
\end{align}
Next we argue that if $\Delta^r\le \frac{{c}_1}{r\sigma_1}$, then we must have
\begin{align}
\Delta^{r+1}\le\frac{{c}_1}{(r+1)\sigma_1}.
\end{align}
Using the condition \eqref{eq:key_recursion2} and the inductive hypothesis $\Delta^r\le \frac{{c}_1}{r\sigma_1}$, we have
\begin{align}
\Delta^{r+1}&\le\frac{-1+\sqrt{1+\frac{4{c}_1}{r}}}{2\sigma_1}
=\frac{2{c}_1}{r\sigma_1\left(1+\sqrt{1+\frac{4{c}_1}{r}}\right)}\nonumber\\
&\le \frac{2{c}_1}{\sigma_1\left(r+\sqrt{r^2+4r+4}\right)}=\frac{{c}_1}{\sigma_1(r+1)}
\end{align}
where the last inequality is due to the fact that ${c}_1\ge 2$, and $r\ge 2$.
Consequently, we have shown that  for all $r\ge 1$
\begin{align}
\Delta^r=f(x^r)-f^*\le \frac{{c}_1}{\sigma_1}\frac{1}{r}.
\end{align}
For the E-C rule, first note that from Lemma \ref{lm:p-descent}, we have
\begin{align}
\Delta^r-\Delta^{r+T}\ge \frac{\gamma}{T K  R^2 G^2_{\rm max}}(\Delta^{r+T})^2:=\sigma_2 (\Delta^{r+T})^2, \; \forall~r\ge 1.
\end{align}
Then using the similar argument as for the G-S rule, we can obtain the desired result.

Next we show part (3) of the claim. For the G-So rule, we have from
Lemma \ref{lemma:clost_to_go}, the second part of Lemma \ref{lm:p-descent}, that for all $r\ge 1$
\begin{align}\label{eq:key_recursion_s_bcd}
&\Delta^r-\Delta^{r+1}\ge \frac{q}{K}{\gamma}\|\hx^{r+1}-x^r\|^2\ge\frac{{\gamma}q}{2K\left((Q+L_h)^2+L^2_{\rm max}KR^2\right)}(\Delta^{r})^2:={\sigma}_3 (\Delta^{r})^2.
\end{align}
Similar relation can be shown for the MBI rule as well. The rest of the proof follows standard argument, see for example \cite[Theorem 1]{nestrov12}.  \QED
\end{proof}

Below we provide further remarks on some special cases of the BSUM algorithm.
\begin{enumerate}
\item One popular choice of the upper bound function $u_k(\cdot,\cdot)$ is \cite{HeLiaoHanYang2002,WangYuan2012, zhang11primaldual, Yang10ADMM,Beck13,nestrov12,shalev11}
\begin{align}\label{eq:particular_bound}
u_k(z_k; x):= g(x)+\langle\nabla_k g(x),  z_k-x_k\rangle+\frac{L_k}{2}\|z_k-x_k\|^2
\end{align}
where the constant  $L_k \ge \rho_{\max}(\nabla^2 g(x))$,  is  often chosen to be largest eigenvalue of the Hessian of $g(x)$.  In this case, evidently we have $\gamma_k=L_k=M_k\le M$, for all $k$, and {$G_{\max}\le M$}. We can also verify that $G_k\le 2M$ for all $k$. Using this choice of $u_k(\cdot;\cdot)$ and $L_k$, the first result in Theorem \ref{thm:complexity} reduces to
\begin{align}
\Delta^r&\le 2\frac{{c}_1 KM^2R^2}{M_{\rm min}}\frac{1}{r}\label{eq:bound_case1}
\end{align}
where $M_{\rm min}:=\min_{k}M_k$. Let us compare the order given in \eqref{eq:bound_case1} with the one stated in \cite[Theorem 6.1]{Beck13}, which is the best known complexity bound for the G-S BCD algorithm for {\it smooth} problems (i.e., when $h_k$ is not present). The bound derived in \cite{Beck13} for smooth constrained problem (resp. smooth unconstrained problem) is in the order of $\frac{KM^2R^2}{M_{\min}}\frac{1}{r}$ (resp. $\frac{M_{\max}K M^2 R^2}{M^2_{\min}}\frac{1}{r}$). These orders are approximately the same as \eqref{eq:bound_case1}. However, our proof covers the general nonsmooth cases, and is simpler. Similarly, when $u_k(\cdot;\cdot)$ takes the form \eqref{eq:particular_bound}, the bounds for the BSUM with the E-C/G-So/MBI rules shown in Theorem~\ref{thm:complexity} can also be simplified.

\item The results derived in Theorem \ref{thm:complexity} is equally applicable to the BCM scheme \eqref{eq:GSperblock} with various block selection rules discussed above. In particular, we can specialize the upper-bound function $u_k$ to be the original smooth function $g$. As long as $g(x_1,\cdots,x_K)$ satisfies the BSC property, Theorem \ref{thm:complexity} carries over. As mentioned in Section \ref{sub:assumptions}, the BSC property is fairly mild and is satisfied in many engineering applications. Nevertheless, we will further relax the BSC condition in the subsequent sections.

\end{enumerate}

\newsection{The BSUM for Single Block Problem}
\subsection{The SUM Algorithm}\label{sub:sum}
In this section, we consider the following single-block problem with $K=1$:
\begin{align}\label{eq:1block_problem}
\min &\quad f(x):=g(x)+h(x)\nonumber\\
\st &\quad x\in X.
\end{align}
In this case the BSUM algorithm reduces to to the so-called successive upper-bound minimization (SUM) algorithm \cite{Razaviyayn12SUM}, listed in the following table.
\begin{center}
\fbox{
\begin{minipage}{5.2in}
\smallskip
\centerline{\bf The Successive Upper-Bound Minimization (SUM) Algorithm}
\smallskip
At each iteration $r+1$, do:
\begin{equation}\label{eq:SUM}
x^{r+1}\in \min_{x\in X}\; u\left(x; x^{r}\right)+h(x).
\end{equation}
\end{minipage}
}
\end{center}
Let us make the following assumptions on the function $u(v; x)$.
\pn {\bf Assumption C.}
\begin{itemize}
\item [(a)]  $u(x; x)= g(x), \quad \forall\; x\in {X}.$
\item [(b)] $u(v; x) \geq g(v),\quad\; \forall\; v \in {X}, \ \forall\; x \in{X}.$
\item [(c)] $\nabla u(x;x)= \nabla g(x), \quad\; \forall\; x\in X$.
\item [(d)] For any given $x$, $u(v; x)$ has Lipschitz continuous gradient, that is
\begin{align}\label{eq:u_lipchitz}
\|\nabla u(v; x)-\nabla u(\hat{v}; x)\|\le L\|v-\hat{v}\|,\ \forall\ \hat{v},\ v\in X, \forall~x\in X,
\end{align}
where $L>0$ is some constant.
\end{itemize}

Compared to Assumption B, Assumption C does not require $u(v;x)$ to be strongly convex in $v$, nor $\nabla u(v; x)$ to be Lipschitz continuous over $x$.
Notice that the Lipschitz continuity of  $\nabla u$ given in \eqref{eq:u_lipchitz} implies the Lipschitz continuity of $\nabla g$.
\begin{proposition}\label{prop:lip}
Suppose $g(x)$ is convex, and $u(v; x)$ satisfies Assumption C. Then we must have
\begin{align}\label{eq:lip:equal}
\|\nabla g(v)-\nabla g(x)\|\le L \|v-x\|,\quad\forall~x,v \in X.
\end{align}
That is, $\nabla g$ is Lipschitz continuous with the coefficient no larger than $L$.
\end{proposition}
\begin{proof}
Utilizing Assumption C, we must have
\begin{align*}
g(v)-g(x) &\le u(v; x) - u(x; x)\\
&\le \langle \nabla u(x;x), v-x\rangle +\frac{L}{2}\|x-v\|^2\\
&= \langle \nabla g(x), v-x\rangle +\frac{L}{2}\|x-v\|^2, \; \forall~x, v\in X.
\end{align*}
Further, using the convexity of $g$ we have
\begin{align*}
g(v)-g(x)\ge \langle \nabla g(x), v-x\rangle, \; \forall~x, v\in X.
\end{align*}
Combining these two inequalities we obtain
\begin{align}\label{eq:g:L}
0\le g(v)-g(x)- \langle \nabla g(x), v-x\rangle \le  \frac{L}{2}\|x-v\|^2, \; \forall~x, v\in X.
\end{align}
%Interchange the role of $x$ and $v$ we have
%\begin{align}
%0\le g(x)-g(v)- \langle \nabla g(v), x-v\rangle \le  \frac{L}{2}\|x-v\|^2.
%\end{align}
%Adding the above two inequalities we have
%\begin{align}\label{eq:g:difference}
%0\le \langle \nabla g(x)-\nabla g(v), x-v\rangle \le {L}\|x-v\|^2
%\end{align}

Similar to \cite[Theorem 2.1.5]{Nesterov04}, we construct the following function
$$\phi(x) = g(x) - \langle\nabla g(v), x\rangle .$$
Clearly $v \in \arg\min \phi(x) $. We have
\begin{align}
\phi(v)\le \phi\left(x-\frac{1}{L}\nabla g(x)\right)\le \phi(x)-\frac{1}{2L}\|\nabla \phi(x)\|^2
\end{align}
where the first inequality is due to the optimality of $v$ and the second inequality uses \eqref{eq:g:L}. Plugging in the definition of $\phi(x)$ and $\phi(v)$ we have
$$g(v)-\langle\nabla g(v), v\rangle\le g(x)-\langle\nabla g(v), x\rangle-\frac{1}{2L}\|\nabla g(v)-\nabla g(x)\|^2.$$
Since the above inequality is true for any $x, v\in X$, we can  interchange $x$ and $v$ and obtain
$$g(x)-\langle\nabla g(x), x\rangle\le g(v)-\langle\nabla g(x), v\rangle-\frac{1}{2L}\|\nabla g(v)-\nabla g(x)\|^2.$$
Adding these two inequalities we obtain
\begin{align*}
\frac{1}{L}\|\nabla g(x)-\nabla g(v)\|^2\le \langle \nabla g(x)-\nabla g(v), x-v\rangle\le \|\nabla g(x)-\nabla g(v)\|\|x-v\|.
\end{align*}
Cancelling $\|\nabla g(x)-\nabla g(v)\|$ we arrive at the desired results.
\QED
\end{proof}

We remark that this result is only true when both $g(\cdot)$ and $u(\cdot; \cdot)$ are convex functions.

 Our main result is that the SUM algorithm converges sublinearly under Assumption C, {\it without} the strong convexity of the upper-bound function $u(v;x)$ in $v$.  The proof of this claim is an extension of Theorem \ref{thm:complexity}, therefore we will only provide its key steps. Observe that the following is true
 \begin{align}
 f(x^r) - f(x^{r+1})&\stackrel{\rm (i)}\ge f(x^r) - \left(u(x^{r+1}; x^r)+h(x^{r+1})\right)\nonumber\\
 &\stackrel{\rm (ii)}\ge f(x^r) - \left(u(\tx^{r+1}; x^r)+h(\tx^{r+1})\right)\stackrel{\rm (iii)}\ge \frac{\gamma}{2}\|x^r- \tx^{r+1}\|^2\label{eq:1block:chain}
 \end{align}
 where $\tx^{r+1}$ is the iterate obtained by solving the following auxiliary problem for any $\gamma>0$
 \begin{align}
 \tx^{r+1}=\arg\min_{x\in X} u(x; x^r)+h(x)+\frac{\gamma}{2}\|x-x^r\|^2.\label{eq:auxiliary:1block}
 \end{align}
In \eqref{eq:1block:chain}, $\rm{(i)}$ is true because $u(x; y)$ is an upper-bound function for $g(x)$ satisfying Assumption C(b); $\rm{(ii)}$  is true because $x^{r+1}$ is a minimizer of problem \eqref{eq:SUM}; $\rm{(iii)}$ is true due to the fact that $\tx^{r+1}$ is the optimal solution of \eqref{eq:auxiliary:1block} while $x^r$ is a feasible solution. 

{
Then we bound $f(x^{r+1})$ using $f(\tx^{r+1})$. We have
\begin{align}
f(x^{r+1})&\le u(x^{r+1};x^r)+ h(x^{r+1})\nonumber\\
&\stackrel{\rm (i)}\le u(\tx^{r+1};x^r)+ h(\tx^{r+1})\nonumber\\
&\stackrel{\rm (ii)}\le u(x^{r};x^r)+\langle \nabla u(x^r; x^r), \tx^{r+1}-x^r \rangle+\frac{L}{2}\|\tx^{r+1}-x^r\|^2 + h(\tx^{r+1})\nonumber\\
&\stackrel{\rm (iii)}\le g(\tx^{r+1})+\langle \nabla u(x^r; x^r), \tx^{r+1}-x^r \rangle+\langle \nabla g(\tx^{r+1}), x^{r}-\tx^{r+1} \rangle+{L}\|\tx^{r+1}-x^r\|^2 + h(\tx^{r+1})\nonumber\\
&\stackrel{\rm (iv)}= g(\tx^{r+1})+\langle \nabla g(\tx^{r+1})-\nabla g(x^r), x^{r}-\tx^{r+1} \rangle+{L}\|\tx^{r+1}-x^r\|^2 + h(\tx^{r+1}) \nonumber\\
&\stackrel{\rm (v)}\le   f(\tx^{r+1}) +{L}\|\tx^{r+1}-x^r\|^2\nonumber
\end{align}
where $\rm (i)$ is due to the optimality of $x^{r+1}$ for problem \eqref{eq:SUM}; $\rm  (ii)$ uses the gradient Lipschitz continuity of $u(\cdot; x^r)$; $\rm  (iii)$ uses the fact that $u(x^r;x^r) = g(x^r)$, the gradient Lipschitz continuity of $g(\cdot)$ derived in Proposition \ref{prop:lip}; $\rm{(iv)}$ uses the fact that $\nabla u(x^r; x^r) = \nabla g(x^r)$ (cf. Assumption C(c)); $\rm(v)$ uses the convexity of $g(\cdot)$. %; $\rm(v)$ uses the first inequality in \eqref{eq:key}, i.e.,
%\begin{align}
%&\langle \nabla g(\tx^{r+1})-\nabla g(x^r), x^{r}-\tx^{r+1} \rangle \le -\frac{1}{L}\|\nabla g(\tx^{r+1})-\nabla g(x^r)\|^2 \le L\|\tx^{r+1}-x^r\|^2.\nonumber
%\end{align}

Utilizing this bound, we derive the estimate of the cost-to-go
\begin{align}
f(x^{r+1})-f(x^*) &\le f(\tx^{r+1})-f(x^*)+{L}\|\tx^{r+1}-x^r\|^2\nonumber\\
&\le \left\langle \nabla g(\tx^{r+1}), \tx^{r+1}-x^*\right\rangle +h(\tx^{r+1})-h(x^*)+{L}\|\tx^{r+1}-x^r\|^2\nonumber\\
&= \left\langle \nabla g(\tx^{r+1})- \nabla g(x^r), \tx^{r+1}-x^*\right\rangle+{L}\|\tx^{r+1}-x^r\|^2 \nonumber\\
&\quad +\left\langle \nabla g(x^{r})- \nabla \left(u(\tx^{r+1}; x^r)+\frac{\gamma}{2}\|\tx^{r+1}-x^r\|^2\right), \tx^{r+1}-x^*\right\rangle \nonumber\\
 &\quad +h(\tx^{r+1})-h(x^*)+\left\langle\nabla \left(u(\tx^{r+1}; x^r)+\frac{\gamma}{2}\|\tx^{r+1}-x^r\|^2\right), \tx^{r+1}-x^*\right\rangle\nonumber\\
&\stackrel{\rm (i)}\le \left\langle\nabla g(\tx^{r+1})- \nabla g(x^r), \tx^{r+1}-x^*\right\rangle +{L}\|\tx^{r+1}-x^r\|^2\nonumber\\
&\quad +\left\langle \nabla u(x^{r}; x^{r})- \nabla u(\tx^{r+1}; x^r), \tx^{r+1}-x^*\right\rangle - \gamma \left\langle \tx^{r+1}-x^r, \tx^{r+1}-x^*\right\rangle\nonumber\\
&\stackrel{\rm (ii)}\le  (2L+\gamma)\|\tx^{r+1}-x^r\|R +{L}\|\tx^{r+1}-x^r\|\|\tx^{r+1}-x^*+x^*-x^r\| \nonumber\\
& \le  (4L+\gamma)\|\tx^{r+1}-x^r\|R \nonumber.
\end{align}
Here $\rm{(i)}$ is due to the optimality of $\tx^{r+1}$ to the problem \eqref{eq:auxiliary:1block}; in $\rm{(ii)}$ we have used \eqref{eq:lip:equal}, Cauchy-Schwartz inequality and the definition of $R$ (it is easy to show that $f(\tx^{r+1})\le f(x^r)\le f(x^0)$, hence $\|\tx^{t+1}-x^*\|\le R$ for all $t$).}

Combining the above two inequalities, we obtain
\begin{align}
\Delta^{r}-\Delta^{r+1}\ge \frac{\gamma}{2R^2 (4L+\gamma)^2}(\Delta^{r+1})^2,\quad \forall \gamma>0.
\end{align}
Maximizing over $\gamma$ (with $\gamma=4L$),  we have
\begin{align}
\Delta^{r}-\Delta^{r+1}\ge \frac{1}{32R^2 L}(\Delta^{r+1})^2:=\sigma_4(\Delta^{r+1})^2.
\end{align}
Using the same derivation as in Theorem \ref{thm:complexity}, we obtain
\begin{align}
\Delta^{r+1}\le \frac{c_4}{\sigma_4}\frac{1}{r}, \quad \mbox{with}\quad \sigma_4=\frac{1}{32R^2 L},\quad c_4:=\max\{4 \sigma_4-2, f(x^1)-f^*, 2\}\label{eq:rate:1block}.
\end{align}

\subsection{Application}

 To see the importance of the above result, consider the well-known method of Iterative Reweighted Least Squares (IRLS) \cite{Beck13b, Daubechies10}. The IRLS is a popular algorithm used for solving problems such as sparse recovery and Fermat-Weber problem; see \cite[Section 4]{Beck13b} for a few applications. Consider the following problem
 \begin{align}
\min_{x}&\quad  h(x) + \sum_{j=1}^{\ell}\|A_j x + b_j\|_2,\quad \st \quad x\in X \label{eq:irls}
 \end{align}
 where $A_j\in\mathbb{R}^{k_i\times m}$, $b_j\in\mathbb{R}^{k_i}$, $X\subseteq \mathbb{R}^m$, and $h(x)$ is some convex function not necessarily smooth. Let us introduce a constant $\eta>0$ and consider a {\it smooth approximation} of problem \eqref{eq:irls}:
\begin{align}
\min_{x}&\quad h(x)+g(x): = h(x) + \sum_{j=1}^{\ell}\sqrt{\|A_j x + b_j\|^2_2+\eta^2},\quad \st \quad  x\in X \label{eq:irls:smooth}.
 \end{align}
  The IRLS algorithm generates the following iterates
 \begin{align}
 x^{r+1}=\arg\min_{x\in X }\left\{h(x)+\frac{1}{2}\sum_{j=1}^{\ell}\frac{\|A_j x+b_j\|^2+\eta^2}{\sqrt{\|A_j x^r+b_j\|^2+\eta^2}}\right\}\label{iteration:irls}.
 \end{align}
It is known that the IRLS iteration is equivalent to a BCM method applied to the following two-block problem (i.e., the first block is $x$ and the second block is $\{z_j\}_{j=1}^{\ell}$)
 \begin{align}\label{eq:twoblock:irls}
\begin{split}
 \min&\quad h(x)+\frac{1}{2}\sum_{j=1}^{\ell}\left(\frac{\|A_j x+b_j\|^2+\eta^2}{z_j}+z_j\right)\\
 \st&\quad x\in X,\quad z_j\in [\eta/2,\infty), \ \forall~j.
 \end{split}
 \end{align}
 Utilizing such two-block BCM interpretation, the author of \cite{Beck13b} shows that the IRLS converges sublinearly when $h(x)$ has Lipschitz continuous gradient; see \cite[Theorem 4.1]{Beck13b}.

 Differently from \cite{Beck13b}, here we take a new perspective. We argue that the IRLS is in fact the SUM algorithm in disguise, therefore our simple iteration complexity analysis given in Section \ref{sub:sum} for SUM can be directly applied.

 Let us consider the following function:
 \begin{align}
 u(x; x^r) = \frac{1}{2}\sum_{j=1}^{\ell}\left(\frac{\|A_j x+b_j\|^2+\eta^2}{\sqrt{\|A_j x^r+b_j\|^2+\eta^2}}+\sqrt{\|A_j x^r+b_j\|^2+\eta^2}\right).\label{bound:irls}
 \end{align}
 It is clear that $g(x^r) = u(x^r; x^r)$, so Assumption C(a) is satisfied. To verify Assumption C(b), we apply the arithmetic-geometric inequality, and have
 \begin{align*}
 u(x; x^r) & = \frac{1}{2}\sum_{j=1}^{\ell}\left(\frac{\|A_j x+b_j\|^2+\eta^2}{\sqrt{\|A_j x^r+b_j\|^2+\eta^2}}+\sqrt{\|A_j x^r+b_j\|^2+\eta^2}\right)\nonumber\\
 &\ge \sum_{j=1}^{\ell}\sqrt{\|A_j x+b_j\|^2+\eta^2} = g(x), \; \forall~x\in X.
 \end{align*}
 Assumptions C(c)-(d) are also easy to verify. Note that the matrices $A_j$'s do not necessarily have full column rank, so $u(x;x^r)$ may not be strongly convex over $x\in X$.  Nevertheless, $u(x;x^r)$ defined in \eqref{bound:irls} is indeed an upper bound function for the smooth function $g(x)$, and we have shown that it satisfies Assumptions C. It follows that the iteration \eqref{iteration:irls} corresponds to a single-block BSUM algorithm. Our analysis leading to \eqref{eq:rate:1block} suggests that this algorithm converges in a sublinear rate, even when $h(x)$ is a nonsmooth function. To be more specific, for this problem we have
 \begin{align}
 %M&=\frac{1}{\eta}\sum_{j=1}^{\ell}\|A_j\|, \quad
 L=\frac{1}{\eta}\rho_{\max}\left(\sum_{j=1}^{\ell}A^T_j A_j\right)\nonumber
 \end{align}

 Therefore the rate can be expressed as
 \begin{align}\Delta^{r+1}\le \max\{4 \sigma_4-2, f(x^1)-f(x^*), 2\}\frac{ 32 R^2 \rho_{\max}\left(\sum_{j=1}^{\ell}A^T_j A_j\right)}{\eta r}.\label{eq:irls:bound}
\end{align}
Note that compared with the result derived in \cite[Theorem 4.1]{Beck13b} which is based on transforming the IRLS algorithm to the two-block BCM problem \eqref{eq:twoblock:irls}, our analysis is based on the key insight of the  equivalence between IRLS and the single block BSUM, and it is significantly simpler. Further we do not require $h(x)$ to be smooth, while the result in \cite[Theorem 4.1]{Beck13b} additionally requires that the gradient of $h(x)$ is Lipschitz continuous \footnote{It appears that the proof in \cite[Theorem 4.1]{Beck13b} can be modified to allow nonsmooth $h$, just that it is not explicitly mentioned in the paper. But as it stands, the bound in \cite[Theorem 4.1]{Beck13b} is explicitly dependent on the Lipschitz constant of the gradient of $h$, while the bound we derived here in \eqref{eq:irls:bound} is not.}.

\newsection{The BSUM for Two Block Problem}
\subsection{Iteration Complexity for 2-Block BSUM}
In this section, we consider the following two-block problem ($K=2$), which is a special case of problem \eqref{eq:bcd_problem}:
\begin{align}\label{eq:2block_problem}
\begin{split}
\min &\quad f(x_1,x_2):=g(x_1,x_2)+h_1(x_1)+h_2(x_2)\\
\st &\quad x_1\in X_1, \ x_2\in X_2.
\end{split}
\end{align}
This problem has many applications, such as the special case of Example \ref{ex:wireless} with two users, the two-block formulation of the IRLS algorithm \eqref{eq:twoblock:irls} or the example presented in \cite[Section 5]{Beck13b}.
Throughout this section, we assume that Assumption A(a) is true. We make the following additional assumptions about problem \eqref{eq:2block_problem}.

 \pn {\bf Assumption D.}
 \begin{itemize}
\item [(a)] The problem $\min_{x_2\in X_2} f(x_1, x_2)$ has a unique solution.
\item [(b)] The gradient of $g(x_1,x_2)$ with respect to $x_1$ is Lipschitz continuous, i.e.,
$$\|\nabla_1 g(x_1, x_2)-\nabla_1 g(v_1, x_2)\|\le M_1\|x_1-v_1\|.$$
 \end{itemize}

Note that here we do not require that the gradient of $g(\cdot)$ with respect to the second block to be Lipschitz continuous.

We first show that for this problem BSUM with G-S update rule is able to achieve sublinear rate without the BSC condition or the Lipschitz continuity of $\nabla_2 g(x_1,x_2)$. Under the same assumption,  we further show that it is possible to accelerate the BSUM method with G-S rule to get an $\mathcal{O}(1/r^2)$ iteration complexity. %The key is to perform exact minimization for one of the blocks.

%\begin{example}
%Consider the following problem, which is the subproblem solved by the so-called Iteratively Reweighted Least Squares Method \cite{Beck13b,Daubechies10}
%\begin{align}\label{eq:IRLS}
%\begin{split}
%\min&\quad s(x_1)+\frac{1}{2}\sum_{i=1}^{n_2}\left(\frac{\|A_i x_1+b_i\|^2+\eta^2}{x_2[i]}+x_2[i]\right)\\
%\st&\quad x_1\in X_1, \; x_2\in [\eta/2, \infty]^{n_2},
%\end{split}
%\end{align}
%where $s(\cdot)$ is a convex possibly nonsmooth function; $X_1\subseteq\mathbb{R}^{n_1}$; $A_i\in\mathbb{R}^{n_2\times n_1}$, $b_i\in\mathbb{R}^{n_2}$ are some input data. Clearly this problem is a special case of problem \eqref{eq:2block_problem}. Further, note that the problem is not necessarily strongly convex w.r.t. $x_1$, but for any $\eta/2$ bounded away from zero, the problem is strongly convex w.r.t. $x_2$.
%\end{example}

In table given below we list the two-block BSUM algorithm with G-S update rule. %To ensure convergence, one can pick {\it any} valid upper-bound function for each block.  %In the following the choice of upper bounds is slightly restricso that we can deal with the absence of BSC and potentially obtain faster rates. %Let $u_1(\hx_1; x_1, x_2)$ be any upper bound for $g(x_1,x_2)$ with respect to the first block that satisfying Assumption B(a)-(c) and the Lipschitz condition \eqref{eq:uk_lipchitz}. That is we do not require $u_1(\hx_1; x_1, x_2)$ to be strongly convex. Further
 %We emphasize that here we do not require $u_1(\hx_1; x_1, x_2)$ to be strongly convex (cf. Assumption D).
%\begin{align}
%\begin{split}\label{eq:2_block_upperbound}
%&u_1(\hx_1; x_1,x_2) &=g(x_1,x_2)+\langle \nabla_1 g(x_1,x_2), \hx_1-x_1\rangle+\frac{M_1}{2}\|\hx_1-x_1\|^2,\\
%u_2(\hx_2; x_1,x_2)&=g(\hx_2,x_1).\\
%\end{split}
%\end{align}
%Clearly $u_2$ is a valid upper-bound function as it is simply the original smooth function $g$ itself.
%
%Applying the first part of Theorem \ref{thm:complexity}, we see that the 2-block BSUM algorithm with G-S rule described in the following table converges in a sublinear rate.
\begin{center}
\fbox{
\begin{minipage}{5.2in}
\smallskip
\centerline{\bf The G-S 2-block BSUM for problem \eqref{eq:2block_problem}}
\smallskip
At each iteration $r+1$, update the variable blocks by:
\begin{align}\label{eq:2_block_original_iteration}
\begin{split}
x_2^{r+1}&=\arg\min_{x_2\in X_2} u_2(x_2; x^r_1, x^r_2)+h_2(x_2)\\
x_1^{r+1}&\in \arg\min_{x_1\in X_1} u_1(x_1; x^{r}_1,x^{r+1}_2)+h_1(x_1).
\end{split}
\end{align}
\end{minipage}
}
\end{center}

Unfortunately for the problem of interest here the rate analysis provided in Theorem \ref{thm:complexity} is no longer applicable because $\nabla_2 g(x_1,x_2)$ may not be Lipschitz continuous, and both subproblems may not be strongly convex. To analyze the convergence rate, let us consider the following special choices of the upper bound where $u_1(x_1; x)$ satisfies Assumption B(a)-(c) and the Lipschtiz continuous gradient condition \eqref{eq:uk_lipchitz}, restated below for convenience
\begin{align}\label{eq:u1_lip}
\|\nabla u_1(x_1; x)-\nabla u_1(v_1; x)\|\le L_1\|x_1-v_1\|, \; \forall~x_1, v_1\in X_1, \ \forall~x\in X.
\end{align}
By utilizing the argument in Proposition \ref{prop:lip}, we can show that $L_1\ge M_1$, therefore the following is true as well
$$\|\nabla_1 g(x_1, x_2)-\nabla_1 g(v_1, x_2)\|\le L_1\|x_1-v_1\|.$$
Further we do not use any upper bound for the second block, i.e.,  we let
$$u_2(v_2; x)=g(v_2,x_1), \; \forall~x_1\in X_1, \ v_2\in X_2.$$ This suggests that the $x_2$-block is minimized exactly.

To analyze the algorithm, it is convenient to consider an equivalent {\it single-block} problem, which only takes $x_1$ as its variable:
\begin{align}\label{eq:2block_singleblock}
\min_{x_1\in X_1}&\quad \ell(x_1)+h_1(x_1):=\min_{x_1\in X_1}\min_{x_2\in {X}_2}{f}(x_1, x_2),
\end{align}
where we have defined $\ell(x_1):=\min_{x_2\in X_2} g(x_1,x_2)+h_2(x_2)$. Let us denote an optimal solution of the  inner problem $\min_{x_2\in {X}_2}{f}(x_1, x_2)$ by the mapping: $x^*_2(x_1): X_1\to X_2$, which is a singleton for any $x_1\in X_1$ by Assumption D(a).
%It is easy to see that problem \eqref{eq:2block_singleblock} is equivalent to problem \eqref{eq:2block_problem}.
Next we analyze problem \eqref{eq:2block_singleblock}.

Let us define a new function
\begin{align}\label{eq:u:2block}
u(v_1; x_1):=u_1(v_1; x_1, x^*_2(x_1))+h_2(x^*_2(x_1)).
\end{align}
First we argue that for all $x_1,v_1\in X_1$, $u(v_1;x_1)$ is an upper bound for $\ell(v_1)$, and it satisfies Assumption C given in Section \ref{sub:sum}. Clearly Assumption C(a) is true because
\begin{align}
\ell(x_1)=g(x_1,x_2^*(x_1))+h_2(x^*_2(x_1)) = u_1(x_1; x_1, x^*_2(x_1))+h_2(x^*_2(x_1))=u(x_1; x_1)
\end{align}
where the second equality is due to the fact that $u_1(x_1; x)$ is an upper bound function for $g(\cdot, x_2)$. The last equality is from the definition of $u(\cdot;\cdot)$.

Assumption C(b) is true because
$$u(v_1 ;x_1)=u_1(v_1; x_1, x^*_2(x_1))+h_2(x^*_2(x_1))\ge g(v_1, x^*_2(x_1))+h_2(x^*_2(x_1))\ge \min_{x_2}g(v_1,x_2)+h_2(x_2).$$
To verify Assumption C(c), recall that by Assumption D the inner problem $\min_{x_2\in {X}_2}{f}(x_1, x_2)$ has a {\it unique} solution, or equivalently  for any given $x_1\in X_1$, the mapping $x^*_2(x_1)$ is a singleton. By applying \cite[Corollary 4.5.2--4.5.3]{Hiriart-Urruty}, we obtain
\begin{align}\label{eq:gradient_consistency}
\nabla \ell (x_1)=\nabla_1 g\left(x_1, \tx_2\right), \; \forall~x_1\in X_1
\end{align}
where $\tx_2=\arg\min_{x_2\in X_2} f(x_1,x_2)$.
Therefore, we must have
$$\nabla \ell(x_1) = \nabla_1 g\left(x_1, \tx_2\right) = \nabla u_1(x_1; x_1, \tx_2)= \nabla u_1(x_1; x_1, x^*_2(x_1))=\nabla u(x_1;x_1),$$
where the second equality comes from the fact that $u_1(\cdot; \cdot)$ satisfies Assumption B(c); the third inequality is because $\tx_2=x^*_2(x_1)$ by definition; the last equality is from \eqref{eq:verify:C1}. This verifies Assumption C(c).

The Lipschitz continuous gradient condition (with constant $L_1$) in Assumption C(d) can be verified by combining \eqref{eq:u1_lip} and the following equality
\begin{align}
\nabla u_1(v_1; x_1, x^*_2(x_1))=\nabla u(v_1;x_1),\; \forall~v_1, x_1\in X_1.\label{eq:verify:C1}
\end{align}

Now that we have verified that $u(v_1; x_1)$ given in \eqref{eq:u:2block} satisfies Assumption C, then Proposition \ref{prop:lip} implies $\ell(\cdot)$ also has Lipschitz continuous gradient with constant $L_1$, that is
$$\|\nabla \ell(x_1)-\nabla \ell(v_1)\|\le L_1\|x_1-v_1\|,\; \forall~v_1, x_1\in X.$$

At this point it is clear that the 2-block BSUM algorithm with G-S update rule is in fact the SUM algorithm given in Section \ref{sub:sum}, where the iterates are generated by
\begin{align}\label{eq:2block:iterate}
x^{r+1}_1\in \arg\min u(x_1; x^r_1).
\end{align}
By applying the argument leading to \eqref{eq:rate:1block}, we conclude that the 2-block BSUM in which the second block performs an exact minimization converges sublinearly. Also note that neither subproblems in \eqref{eq:2block_problem} is required to be strongly convex, which suggests that the BCM applied to problem \eqref{eq:2block_problem} converges sublinearly without block strong convexity. The precise statement is given in the following corollary.

\begin{corollary}\label{cor:2-block}
Assume that Assumption A(a) and D hold for problem \eqref{eq:2block_problem}. Then we have the following.
\begin{enumerate}
\item Suppose that $u_2(v_2; x)=g(x_1,v_2)$ for all $v_2\in X_2, \ x\in X$ and that $u_1(v_1; x)$ satisfies Assumption B(a)-(c) and the Lipschtiz continuous gradient condition \eqref{eq:uk_lipchitz}. Then the 2-block BSUM algorithm with G-S rule is equivalent to the SUM algorithm and converges sublinearly, i.e.,
    \begin{align}
\Delta^{r+1}\le \frac{c_4}{\sigma_4}\frac{1}{r}
\end{align}
where $c_4$ and $\sigma_4$ is given in \eqref{eq:rate:1block}, with $L$ in \eqref{eq:rate:1block} replaced by $L_1$.
\item The BCM algorithm applied to \eqref{eq:uk_lipchitz} converges sublinearly with the same rate, again with $L$ in \eqref{eq:rate:1block} replaced by $L_1$.
\end{enumerate}

\end{corollary}

\subsection{Accelerating the 2-Block BSUM}

Next we show that it is possible to accelerate the above G-S BSUM iterations \eqref{eq:2_block_original_iteration} to obtain an improved rate. The main idea is again to use the {\it single-block} interpretation of the 2-block BSUM.

Let us pick the following upper bound function for $x_1$
\begin{align}\label{eq:u_1:acc}
u_1(v_1; x) &= \langle \nabla_1 g(x_1, x_2), v_1-x_1\rangle + h_1(v_1)+\frac{M_1}{2}\|v_1-x_1\|^2,
\end{align}
where $M_1$ is the Lipschitz constant for $\nabla_1 g(x_1, x_2)$.
Then utilizing the single block interpretation of the 2-block BSUM \eqref{eq:2block:iterate} we must have
\begin{align}
x^{r+1}_1&= \arg\min_{x_1\in X_1} u(x_1; x^r_1)=\arg\min_{x_1\in X_1} u_1(x_1; x^r_1, x^{r+1}_2)+h_1(x_1)\nonumber\\
&=\prox^{M_1}_{h_1+I_{X_1}}\left[ x^r_1 -\frac{1}{M_1}\nabla g(x^r_1, x^{r+1}_2)\right]\nonumber\\
&= \prox^{M_1}_{h_1+I_{X_1}}\left[ x^r_1 -\frac{1}{M_1}\nabla \ell(x^r_1)\right]
\end{align}
where the last inequality comes from \eqref{eq:gradient_consistency}.

%Next, we relate the iterates $(x^{r}_1, x^r_2)$ generated by \eqref{eq:2_block_original_iteration} to the function $\ell(\cdot)$. Let us set the constant $C$ in the iteration \eqref{eq:2_block_original_iteration} to $C=\tM>M_1$. Then we have
%\begin{align}
%x^{r+1}_1=\prox_{h_1+I_{X_1}}\left[x_1^r-\frac{1}{\tM} \nabla_1 g(x^r_1, x^{r+1}_2)\right]\stackrel{\eqref{eq:gradient_consistency}}=\prox_{h_1+I_{X_1}}\left[x_1^r-\frac{1}{\tM} \nabla \ell(x^r_1)\right].
%\end{align}
%That is, each iteration of the two-blcok BSUM algorithm \eqref{eq:2_block_original_iteration} is equivalent to performing a single proximal gradient step for the single-block problem \eqref{eq:2block_singleblock}.
This observation lends itself to a simple acceleration scheme by applying known Nesterov-type acceleration schemes. The scheme, named Accelerated 2-Block BSUM (A-2BSUM) Algorithm, is described in the following table. The $\mathcal{O}(1/r^2)$ iteration complexity of the algorithm can be obtained directly from existing analysis for accelerated proximal gradient; see, e.g.,
\cite{Beck:2009:FIS:1658360.1658364, tseng08acc, Nesterov04}. It is interesting to see that for the two block problem \eqref{eq:2block_problem}, the acceleration scheme developed here as well as the resulting rate are not dependent on the Lipschitz constant for the gradient of the second block, since we do not require $\nabla_2 g(x_1,x_2)$ to be Lipschitz continuous.
\begin{center}
\fbox{
\begin{minipage}{5 in}
\smallskip
\centerline{\bf The A-2BSUM Algorithm}
\smallskip
At any given iteration $r>1$, do the following:\\
S1) Choose $\theta^r=\frac{2}{r+1}$;\\
S2) ${v}_1^{r}=(1-\theta^{r-1}){x}^{r-1}_1+\theta^r({w}^{r-1}_1)$;\\
S3) $x^{r}_2=\arg\min_{{x}_2\in X_2} {f}({v}^r_1, x_2)$;\\
S4) $x_1^{r}=\arg\min_{{x}_1\in X_1} u_1(x_1; v^r_1, x^r_2)$, where $u_1$ is given in \eqref{eq:u_1:acc};\\
S5) ${w}^{r}_1=x_1^{r-1}+\frac{1}{\theta^r}(x^{r}_1-x^{r-1}_1)$.\\
\end{minipage}
}
\end{center}

To conclude this section, we note that the schemes and analysis developed in this section are special in the sense that they heavily rely on the fact that $K=2$, and the resulting transformation to the single block problem. It is unclear whether the same sublinear iteration complexity holds for a general $K$ without the BSC condition, or if the algorithm can be accelerated for any $K$; see \cite{tseng08acc, Beck13} for related discussions.

\newsection{Analysis of the BCM without Per-Block Strong Convexity}\label{sec:BSC}
In this section, we consider the BCM algorithm below, which is the BSUM algorithm without using approximation for each block. We analyze its iteration complexity {\it without} the BSC assumption.
\begin{center}
\fbox{
\begin{minipage}{5.2in}
\smallskip
\centerline{\bf The Block Coordinate Minimization (BCM) Algorithm}
\smallskip
At each iteration $r+1$, pick an index set $\cC^{r+1}$; update the variable blocks by:
\begin{equation}\label{eq:BCM}
x^{r+1}_k
\left\{ \begin{array}{ll}\in \min_{x_k\in X_k}\; g\left(x_k, w^{r+1}_{-k}\right)+h_k(x_k), & \mbox{if}\; k\in\cC^{r+1};\\
=x^{r}_k, &\mbox{if}\; k\notin\cC^{r+1}.\nonumber
\end{array}\right.\\
\end{equation}
\end{minipage}
}
\end{center}

{In the absence of the BSC property, there can be multiple optimal solutions for each subproblem. This makes it tricky to establish the convergence of BCM. Specifically, in the context of the three-step analysis framework presented herein, it is difficult to bound the sufficient descent of the objective using the size of of the successive iterates (as per Lemma \ref{lm:p-descent}). In this section, we overcome this obstacle by developing several variants of the sufficient descent estimate step. We first show that BCM with MBI, G-S and E-C rules has an iteration complexity of $O(1/r)$ for problem \eqref{eq:bcd_problem} without the BSC condition. Further, we argue that for certain special classes of problem \eqref{eq:bcd_problem}, this sublinear rate can be improved in terms of the dependence on $K$ for the G-S/E-C rules. Throughout this section we will impose Assumption A.}

We first consider the MBI rule. We notice that the following is true
\begin{align}
f(x^{r})-f(x^{r+1})\stackrel{\rm (i)}\ge f(x^r)-f(\bar{x}^{r+1})\stackrel{\rm (ii)}\ge\frac{\gamma}{K}\|x^r-\hx^{r+1}\|^2,
\end{align}
{where $\bar{x}^{r+1}$ is the iterates obtained by any BSUM algorithm with MBI rule; $\hx^{r+1}$ is defined in \eqref{eq:def_hatx}. In the above expression $\rm(ii)$ can be obtained using Lemma \ref{lm:p-descent}, while $\rm(i)$ is true because we used the exact minimization in each step}. Then it is straightforward to establish, using the additional assumption that $h$ is Lipschitz continuous, the same rate stated in part (3) of Theorem \ref{thm:complexity}.

{Next we show that the BCM algorithm with the G-S and E-C rules also achieves an $\mathcal{O}(1/r)$ iteration complexity, without the BSC assumption.}

\subsection{A General Analysis for G-S and E-C rules}

The main difficulty in analyzing the BCM without the BSC is that the size of the difference of the successive iterates is no longer a good measure of the ``sufficient descent". Indeed, due to the lack of per-block strong convexity, it is possible that a block variable travels a long distance without changing the objective value (i.e., it stays in the per-block optimal solution set).

%A consequence of this observation is that the classical G-So update rule \eqref{eq:G-So} should be modified.  In this section, we consider the following {\it modified} G-So rule, where the best index is selected based on the maximum difference of the size of the gradients before and after changing a single block:
%\begin{itemize}
%\item {\it Modified Gauss-Southwell (M-G-So) update rule}: At each iteration $r+1$, $\cC^{r+1}$ contains a single index $k^*$ that satisfies:
%\begin{align}\label{eq:modified_G_So}
%k^*\in \left\{k \;\bigg{|}\; \|\nabla g(\hx^{r+1}_{k}, x^r_{-k})-\nabla g(x^{r}_{k}, x^r_{-k})\|\ge q \max_{j}\|\nabla g(\hx^{r+1}_j, x^r_{-j})-\nabla g(x^{r}_j, x^r_{-j})\|\right\}
%\end{align}
%for some constant $q\in (0,\; 1]$.
%\end{itemize}
%Such modification makes perfect sense -- when no BSC condition is assumed, the block that travels the furthest is not necessarily a ``good block" anymore. Therefore using the successive difference of the iterates to gauge the progress of the algorithm may no longer be meaningful.

Below we analyze the iteration complexity of BCM. We need to make use of the following key inequality due to Nesterov \cite{Nesterov04}; also see \eqref{eq:g:L} for a proof. From Assumption A we know that $g$ is convex and has Lipschitz continuous gradient with constant $M$, then we must have have
\begin{align}\label{BCM_Lipschitz_Modified}
g(x)-g(v)\ge \langle \nabla g(v), x-v \rangle+\frac{1}{2M}\|\nabla g(v)-\nabla g(x)\|^2, \; \forall~v,x\in X.
\end{align}
%Similarly, using the Lipschitz continuity assumption on $\nabla g(x_k, x_{-k})$ (with constant $M_k$), we have
%\begin{align}\label{BCM_Lipschitz_Modified2}
%&g(x_k, x_{-k})-g(v_k, x_{-k})\nonumber\\
%&\ge \langle \nabla_k g(v_k, x_{-k}), v_k-x_k \rangle+\frac{1}{2M_k}\|\nabla_k g(v_k, x_{-k})-\nabla_k g(v_k, x_{-k})\|^2, \; \forall~x_k, v_k\in X_k.
%\end{align}

Utilizing this inequality, the sufficient descent estimate is given by the following lemma.

\begin{lemma}
Suppose Assumption A holds. Then for BCM with either G-S rule or the E-C rule, we have that for all $r\ge 1$
\begin{equation}
\Delta^r - \Delta^{r+1} \ge \frac{1}{2M}\sum_{k=1}^K  \| \nabla g(w_k^{r+1}) - \nabla g(w_{k+1}^{r+1}) \|^2.
\end{equation}
%\item For the BCM with the modified G-So rule, we have that for all $r\ge 1$,
%\begin{equation}
%\Delta^r - \Delta^{r+1} \ge  \frac{q}{2KM}\sum_{k=1}^K  \| \nabla g(w_k^{r+1}) - \nabla g(\hat{x}^{r+1}_k,w_{-k}^{r+1}) \|^2.
%\end{equation}
%\end{enumerate}
\end{lemma}

{\bf{Proof.}} Suppose that $k\notin \cC^{r+1}$, then we have the following trivial inequality
\begin{equation}
f(w_k^{r+1}) - f(w_{k+1}^{r+1}) \ge \frac{1}{2M} \| \nabla g(w^{r+1}_k)-\nabla g(w^{r+1}_{k+1}) \|^2
\end{equation}
as both sides of the inequality are zero.

Suppose $k\in \cC^{r+1}$. Then by \eqref{BCM_Lipschitz_Modified}, we have that
\begin{align}
& f(w^{r+1}_{k})-f(w^{r+1}_{k+1})\nonumber \\
& \ge \langle \nabla g(w^{r+1}_{k+1}), w^{r+1}_{k}- w^{r+1}_{k+1}\rangle+h(x_k^{r})-h(x_k^{r+1})+\frac{1}{2M}\|\nabla g(w^{r+1}_k)-\nabla g(w^{r+1}_{k+1}) \|^2\nonumber\\
& \stackrel{\rm(i)}\ge \langle \nabla_k g(w^{r+1}_{k+1}), x^{r}_{k}- x^{r+1}_k\rangle+h_k(x_k^{r})-h_k(x_k^{r+1})+\frac{1}{2M}\|\nabla g(w^{r+1}_k)-\nabla g(w^{r+1}_{k+1}) \|^2\nonumber\\
&\stackrel{\rm(ii)}\ge\frac{1}{2M}\|\nabla g(w^{r+1}_k)-\nabla g(w^{r+1}_{k+1}) \|^2
\end{align}
where $\rm{(i)}$ is because $w^{r+1}_{k+1}$ and $w^{r+1}_k$ only differs by a single block; $\rm {(ii)}$ is due to the optimality of $x^{t+1}_{k}$.
Summing over $k$, we have
\begin{equation}
f(x^r)-f(x^{r+1}) \ge  \sum_{k=1}^K \frac{1}{2M}\| \nabla g(w_k^{r+1}) - \nabla g(w_{k+1}^{r+1}) \|^2.
\end{equation}
This completes the proof of this lemma. \hfill {\bf{Q.E.D.}}

\begin{lemma}\label{lemma:cost-to-to_bcm}
Suppose Assumptions A is satisfied. Then
\begin{enumerate}
\item For the BCM with the G-S update rule, we have
$$
(\Delta^{r+1})^2 \le 2 K^2 R^2 \sum_{k=1}^{K}\|\nabla g(w_{k+1}^{r+1})-\nabla g(w^{r+1}_{k})\|^2, \; \forall x^*\; \in X^*.
$$
\item For the BCM with the period-T E-C update rule, we have
$$
(\Delta^{r+T})^2 \le 2 T K^2 R^2 \sum_{k=1}^{K}\sum_{t=1}^{T}\|\nabla g(w_{k+1}^{r+t})-\nabla g(w^{r+t}_{k})\|^2, \; \forall\; x^* \in X^*.
$$
\end{enumerate}
\end{lemma}

{\bf{Proof.}} %We first show the first part of the claim. We have the following series of inequalities
%\begin{align}
%&f(x^{r+1})-f(x^*)\nonumber\\
%&\le \langle \nabla g(x^{r+1}), x^{r+1}-x^*\rangle +\sum_{k=1}^{K}h_k(x_k^{r+1})-h_k(x_k^{*})\nonumber\\
%%&=\sum_{k=1}^{K} \langle \nabla_k g(x^{r+1}), x_k^{r+1}-x_k^*\rangle +\sum_{k=1}^{K}h_k(x_k^{r+1})-h_k(x_k^{*})\nonumber\\
%&=\sum_{k=1}^{K} \langle \nabla_k g(x^{r+1})-\nabla_k g(w^{r+1}_{k+1}), x_k^{r+1}-x_k^*\rangle + \langle \nabla_k g(w^{r+1}_{k+1}), x_k^{r+1}-x_k^*\rangle +h_k(x_k^{r+1})-h_k(x_k^{*})\nonumber\\
%&\stackrel{\rm (i)}\le \sum_{k=1}^{K} \langle \nabla_k g(x^{r+1})-\nabla_k g(w^{r+1}_{k+1}), x_k^{r+1}-x_k^*\rangle \nonumber\\
%&\le \sum_{k=1}^{K} \|\nabla_k g(x^{r+1})-\nabla_k g(w^{r+1}_{k+1})\| \|x_k^{r+1}-x_k^*\|\nonumber\\
%&\le \sum_{k=1}^{K} \sum_{j=1}^{K}\|\nabla_k g(w_{j+1}^{r+1})-\nabla_k g(w^{r+1}_{j})\| \|x_k^{r+1}-x_k^*\|\nonumber\\
%&\le R\sum_{k=1}^{K} \sum_{j=1}^{K}\|\nabla_k g(w_{j+1}^{r+1})-\nabla_k g(w^{r+1}_{j})\| \nonumber
%%&=\sum_{j=1}^{K}\|\nabla g(w_{j+1}^{r+1})-\nabla g(w^{r+1}_{j})\| \sum_{k=1}^{K} \|x_k^{r+1}-x_k^*\|\nonumber
%\end{align}
%where in $\rm (i)$ again we have sued the optimality condition of $x^{r+1}_k$.
%Therefore
%\begin{align}
%(f(x^{r+1})-f(x^*))^2\le  K^2 R^2 \sum_{k=1}^{K}\|\nabla_k g(w_{k+1}^{r+1})-\nabla_k g(w^{r+1}_{k})\|^2
%\end{align}
We only show the second part of the claim, as the proof for the first part is simply a special case. Define a new index set $\{r_k\}$ as in \eqref{eq:r_k}. %:%\begin{equation}r_k := \arg\max_t \{ x_k^t \ne x_k^{r+T} \} +1, k=1,\cdots,K.\end{equation}
%That is, $r_k$ is the latest iteration index (up until $r+T$) in which the $k$-th variable has been updated.
Recall that we have $x_k^{r_k} = x_k^{r+T}$, for all $k$.
We have the following series of inequalities
\begin{align}
&f(x^{r+T})-f(x^*)\nonumber\\
%&\le \langle \nabla g(x^{r+T}), x^{r+T}-x^*\rangle +\sum_{k=1}^{K} h_k(x_k^{r_k}) - h_k(x_k^{*})\nonumber\\
&\le\sum_{k=1}^{K} \langle \nabla_k g(x^{r+T}), x_k^{r+T}-x_k^*\rangle +\sum_{k=1}^{K} h_k(x_k^{r_k})-h_k(x_k^{*})\nonumber\\
&=\sum_{k=1}^{K} \langle \nabla_k g(x^{r+T}) - \nabla_k g(w^{r_k}_{k+1}), x_k^{r+T}-x_k^*\rangle + \langle \nabla_k g(w^{r_k}_{k+1}), x_k^{r+T}-x_k^*\rangle +h_k(x_k^{r_k})-h_k(x_k^{*})\nonumber\\
&\stackrel{\rm (i)}\le \sum_{k=1}^{K} \langle \nabla_k g(x^{r+T})-\nabla_k g(w^{r_k}_{k+1}), x_k^{r+T}-x_k^*\rangle \nonumber\\
&\le \sum_{k=1}^{K} \|\nabla g(x^{r+T})-\nabla g(w^{r_k}_{k+1})\| \|x_k^{r+T}-x_k^*\|\nonumber\\
&\le \sum_{k=1}^{K} \sum_{t=1}^T \sum_{j=1}^K \|\nabla g(w^{r+t}_{j+1})-\nabla g(w^{r+t}_{j})\| \|x_k^{r+T}-x_k^*\|\nonumber\\
&\le \sum_{t=1}^T  \sum_{j=1}^K\|\nabla g(w_{j+1}^{r + t })-\nabla g(w^{r + t}_{j})\| \sum_{k=1}^{K} \|x_k^{r+T}-x_k^*\|\nonumber
%&\le \sum_{k=1}^{K} \left( \|\nabla_k g(x^{r+T})-\nabla_k g(x^{r_k})\| + \|\nabla_k g(x^{r_k})-\nabla_k g(w^{r_k}_k)\|\right) \|x_k^{r+T}-x_k^*\|\nonumber
%&\le \sum_{k=1}^{K} \sum_{j=1}^{K}\|\nabla_k g(w_{j+1}^{r+1})-\nabla_k g(w^{r+1}_{j})\| \|x_k^{r+1}-x_k^*\|\nonumber\\
%&\le \sum_{k=1}^{K} \sum_{j=1}^{K}\|\nabla g(w_{j+1}^{r+1})-\nabla g(w^{r+1}_{j})\| \|x_k^{r+1}-x_k^*\|\nonumber\\
%&=\sum_{j=1}^{K}\|\nabla g(w_{j+1}^{r+1})-\nabla g(w^{r+1}_{j})\| \sum_{k=1}^{K} \|x_k^{r+1}-x_k^*\|\nonumber
\end{align}
where in ${\rm (i)}$ we have used the optimality of $x_k^{r_k}$ and $x_k^{r_k} = x_k^{r+T}$, for all $k$.
Then taking the square on both sides, we obtain
\begin{align}
&\left( f(x^{r+T})-f(x^*) \right)^2
{{\le  T K^2 R^2 \sum_{t=1}^T \sum_{k=1}^K \| \nabla g(w_{k+1}^{r+t})-\nabla g(w_k^{r+t}) \|^2}}.
%&\le \sum_{k=1}^{K} \sum_{j=1}^{K}\|\nabla_k g(w_{j+1}^{r+1})-\nabla_k g(w^{r+1}_{j})\| \|x_k^{r+1}-x_k^*\|\nonumber\\
%&\le \sum_{k=1}^{K} \sum_{j=1}^{K}\|\nabla g(w_{j+1}^{r+1})-\nabla g(w^{r+1}_{j})\| \|x_k^{r+1}-x_k^*\|\nonumber\\
%&=\sum_{j=1}^{K}\|\nabla g(w_{j+1}^{r+1})-\nabla g(w^{r+1}_{j})\| \sum_{k=1}^{K} \|x_k^{r+1}-x_k^*\|\nonumber
\end{align}
The proof is complete. \QED

Combining these two results, and utilizing the technique in Theorem \ref{thm:complexity}, we readily have the following main result for BCM.

\begin{theorem}\label{thm:complexity:bcm}
Suppose Assumption A holds true. We have the following.
\begin{enumerate}
\item Let $\{x^r\}$ be the sequence generated by the BCM algorithm with G-S rule. Then we have
\begin{align}
\Delta^r=f(x^r)-f^*\le \frac{c_5}{\sigma_5}\frac{1}{r}, \; \forall~r\ge 1,
\end{align}
where the constants are given below
\begin{align}
{c}_5&=\max\{4\sigma_5-2, f(x^1)-f^*,2\}\nonumber,\\
\sigma_5&=\frac{1}{2MK^2R^2},\quad \label{eq:sigma5}
\end{align}
\item Let $\{x^r\}$ be the sequence generated by the BCM algorithm with E-C rule. Then we have
\begin{align}
\Delta^r=f(x^r)-f^*\le \frac{c_6}{\sigma_6}\frac{1}{r-T}, \; \forall~r>T,
\end{align}
where the constants are given below
\begin{align}
{c}_6&=\max\{4\sigma_6-2, f(x^1)-f^*,2\}\nonumber,\\
\sigma_6&=\frac{1}{ 2 K^2 T R^2 M}.
\end{align}
\end{enumerate}

\end{theorem}

\subsection{Special Case: The Constrained Nonsmooth Composite Problem}
{The rate derived in the previous subsection is inversely proportional to $K^2$, which is worse than most of the rates derived so far for problems with the BSC assumption. In the following two subsections we sharpen the above results for two special problems of \eqref{eq:bcd_problem}.}

We first make the following additional assumption on the smooth part of the problem \eqref{eq:bcd_problem} (besides Assumption A). Suppose that $g(x)$ takes the following form
\begin{align}\label{eq:g_composite}
g(x)&=\sum_{i=1}^{I}g^i(x_1, x_2,\cdots, x_K)+\sum_{k=1}^{K}b^T_k x_k\nonumber\\
&=\sum_{i=1}^{I}\ell^i(A^i_{1}x_1, A^i_{2}x_2,\cdots, A^i_{k} x_K)+\sum_{k=1}^{K}b^T_k x_k
\end{align}
where the smooth function $\ell^i(y^{i}_{1},\cdots, y^{i}_{K})$ is a {\it composite} of a strongly convex function and a linear mapping. Specifically, $\ell^i(\cdot)$ satisfies the following conditions.
 \pn {\bf Assumption E.}
\begin{itemize}
\item[(a)] $\ell^i(\cdot)$ is strongly convex with respect to each block variable $y^i_{k}$, with $\eta^i_{k}$ as the modulus.
\item[(b)] $\nabla_k\ell^i(y^i_{k},y^i_{-k})$ is Lipschitz continuous for all feasible $y^i_{k}$,
    \begin{align}
    \|\nabla_k\ell^i(y^i_{k},y^i_{-k})-\nabla_k\ell^i(y^i_k,\tilde{y}^i_{-k})\|\le P^i_k\|y^i_{-k}-\tilde{y}^i_{-k}\|,\quad\forall~y^i_k, \; y^i_{-k}, \; \tilde{y}^i_{-k},
    \end{align}
where $P^i_{k}$ is the Lipschitz constant.
%\item[(c)] {$A^i_k\neq 0$ for all $k$ and $i$.} %$\|A^i_k (A^i_k)^T\|$ is bounded away from zero. }
   % This condition can be implied by the Lipchitz continuity condition of $\nabla \ell^i(\cdot)$.
\end{itemize}
Note that the smooth part $g(x)$ may not be strongly convex with respect to any block $x_k$, as $A^i_k$'s can be rank deficient. Two simple examples covered by this family of problems are provided below.

\begin{example}
The sparse logistic regression (SLR) problem with a compact feasible set and the group LASSO problem are special cases of problem \eqref{eq:bcd_problem} with composite smooth function in the form of \eqref{eq:g_composite}. More specifically, these problems have the following objective functions, respectively:
\begin{align}
f^{\rm \mbox{SLR}}(x)&=\sum_{i=1}^{I}\log\left(1+\exp(-y_i a^T_i x)\right)+\nu \|x\|_1,\nonumber\\
f^{\rm \mbox{G-LASSO}}(x)&=\left\|\sum_{k=1}^{K}A_k x_k-b\right\|^2+\sum_{k=1}^{K}\nu_k \|x_k\|_2.\nonumber
\end{align}
For the SLR problem, $\nu\ge 0$ is the penalty coefficient; $I$ is the total number of observations; $y_i\in\mathbb{R}$ is the $i$-th observation; $a_i\in\mathbb{R}^{n}$ is the $i$-th data point. For the group LASSO problem, $\{\nu_k\ge0\}$ are the penalty coefficients; each $A_k\in\mathbb{R}^{m\times n}$ is a data matrix not necessarily having full column rank; and $b\in\mathbb{R}^{m}$ is the observation vector.
\end{example}

%\begin{example}\label{ex:wireless2}
%Consider a generalized version of the rate maximization problem discussed in Example \ref{ex:wireless}, where both the users and the BS have multiple antennas. Suppose $K$ users transmit to a single base station (BS) in the network. Each user is equipped with $N$ transmit antennas and the BS is equipped with $M$ receive antenna. Let $V_k\in \mathbb{S}_{+}^{N}$ denote user $k$'s transmit covariance; $P_k$ denote the maximum transmit power for user $k$; $H_k\in\mathbb{C}^{M\times N}$ denote the channel between user $k$ and the BS. Then the uplink channel capacity optimization problem is given by the following convex program \cite{yu04}
% \begin{align*}
% \max_{p_k}\; \log{\det}\left(\sum_{k=1}^{K}{H_k  V_k H^T_k} +I\right), \quad \st\quad V_k\succeq 0, \; \trace\left[V_k\right]\le P_k, \ k=1,\cdots, K.
% \end{align*}
% Clearly this problem is BSC, as the magnitude of the channel $|h_k|$ is always bounded away from zero.
%\end{example}

Our analysis consists of similar three main steps as before.
To simplify presentation, below we only show the analysis and result for BCM with G-S update rule.
We first show the sufficient descent property. By using the short-handed notation:
$$A^{i}_{-k}w^{r+1}_{-k}:=[A^{i}_1 x^{r+1}_1, \cdots, A^{i}_{k-1} x^{r+1}_{k-1}, A^{i}_{k+1} x^{r}_{k+1}, \cdots A^{i}_{K} x^{r}_{K}], $$
we have the following series of inequalities
\begin{align}\label{eq:sufficient_descent_BCM_GS}
&f(x^{r}_k, w^{r+1}_{-k})-f(x^{r+1}_k, w^{r+1}_{-k})\nonumber\\
&\stackrel{{\rm (i)}}\ge \sum_{i=1}^{I}\left(\langle\nabla_k \ell^{i}(A^i_kx^{r+1}_k, A^{i}_{-k}w^{r+1}_{-k}), A^i_k (x^{r}_k-x^{r+1}_k)\rangle+\frac{\eta^i_k}{2}\|A^i_k(x^{r+1}_k-x^r_k)\|^2\right)\nonumber\\
&\quad\quad\quad+\langle b_k, x^r_k-x^{r+1}_k\rangle+h_k(x^r_k)-h_{k}(x^{r+1}_k)\nonumber\\
&\stackrel{{\rm (ii)}}= \langle\nabla_k g(w^{r+1}_{k+1}), x^{r}_k-x^{r+1}_k\rangle+h_k(x^r_k)-h_{k}(x^{r+1}_k)+\sum_{i=1}^{I}\frac{\eta^i_k}{2}\|A^i_k(x^{r+1}_k-x^r_k)\|^2\nonumber\\
&\stackrel{{\rm (iii)}}\ge \sum_{i=1}^{I}\frac{\eta^i_k}{2}\|A^i_k(x_k^r-x_k^{r+1})\|^2,\nonumber
\end{align}
where in ${{\rm (i)}}$ we have used the strong convexity property of $\ell^i(\cdot)$; in ${{\rm (ii)}}$ we have used the property that $\nabla_k g^i(w^{r+1}_{k+1})=(A^i_k)^T\nabla_k \ell^{i}(A^i_kx^{r+1}_k, A^{i}_{-k}w^{r+1}_{-k})$; in ${{\rm (iii)}} $ we have used the optimality of $x^{r+1}_k$.

As a result we have
\begin{align}
f(x^{r})-f(x^{r+1})\ge \frac{1}{2}\sum_{i=1}^{I}\sum_{k=1}^{K}\eta^i_k\|A^i_k(x_k^r-x_k^{r+1})\|^2.
\end{align}
It is important to note that the sufficient descent estimate described above is measured by the size of the linearly transformed version of the successive difference of the iterates, as opposed to the size of the successive difference of the iterates given in Lemma \ref{lm:p-descent}.

Next let us show the cost-to-go estimate. First note that when $g(x)$ is the composite function described above, we have
\begin{align}
&\|\nabla_k g(x^{r+1})-\nabla_k g(x_k^{r+1}, w^{r+1}_{-k})\|\nonumber\\
&=\left\|\sum_{i=1}^{I}(A_k^i)^T \left(\nabla_k\ell^i(A^i_1 x_1^{r+1},\cdots, A^i_K x_K^{r+1})-\nabla_k\ell^i(A^i_k x^{r+1}_k, A^{i}_{-k}w^{r+1}_{-k})\right)\right\|\nonumber\\
&\le \sum_{i=1}^{I}\sqrt{\left\|A^i_k (A^i_k)^T\right\|}P^i_k\sqrt{\sum_{j=1}^{K}\|A^i_j(x_j^{r+1}-x_j^{r})\|^2}.\nonumber
\end{align}

We have the following series of inequalities
\begin{align}\label{eq:cost_to_go_BCM_GS_1}
f(x^{r+1})-f(x^*)&\le \langle\nabla g(x^{r+1}), x^{r+1}-x^*\rangle+h(x^{r+1})-h(x^*)\nonumber\\
&=\sum_{k=1}^{K}\langle \nabla_k g(x^{r+1})-\nabla_k g(x^{r+1}_k, w^{r+1}_{-k}), x^{r+1}_k-x^*_k\rangle\nonumber\\
&\quad\quad{+\sum_{k=1}^{K}\langle \nabla_k g(x^{r+1}_k, w^{r+1}_{-k}), x^{r+1}_k-x^*_k\rangle+h(x^{r+1})-h(x^*)}\nonumber\\
&\le \sum_{k=1}^{K}\langle \nabla_k g(x^{r+1})-\nabla_k g(x^{r+1}_k, w^{r+1}_{-k}), x^{r+1}_k-x^*_k\rangle\nonumber\\
&\le \sum_{i=1}^{I}\sqrt{\sum_{j=1}^{K}\|A^i_j(x_j^{r+1}-x_j^{r})\|^2}\sum_{k=1}^{K}\sqrt{\left\|A^i_k (A^i_k)^T\right\|}P^i_k\|x^{r+1}_k-x^*_k\|
%&= \sqrt{\sum_{j=1}^{K}\|A^i_j(x_j^{r+1}-x_j^{r})\|^2}\sum_{k=1}^{K}\sum_{i=1}^{I}\sqrt{\|A^i_k (A^i_k)^T\|}P^i_k\|x^{r+1}_k-x^*_k\|.
\end{align}
{where the second inequality is true due to the optimality of $x^{r+1}_k$.}

%Define
%\begin{align}\label{eq:RL}
%R_L:=\max_{x\in X}\max_{x^*\in X^*}\left\{\sqrt{\sum_{k=1}^{K}\sum_{i=1}^{I}{\rho\left(A^i_k (A^i_k)^T\right)}(P^i_k)^2\|x_k-x^*_k\|^2}: f(x)\le f(x^1)\right\}
%\end{align}

Squaring both sides of \eqref{eq:cost_to_go_BCM_GS_1} we obtain
\begin{align}
{(f(x^{r+1})-f(x^*))^2}\le  K I R^2 \max_{k,i}\|A^i_k (A^i_k)^T\|_2(P^i_k)^2 \sum_{i=1}^{I}\sum_{k=1}^{K}\|A^i_k(x^{r+1}_k-x^{r}_k)\|^2. \label{eq:cost_to_go_BCM_GS}
\end{align}

Then by the similar argument as in Theorem \ref{thm:complexity}, we have the following result.

\begin{corollary}\label{cor:complexity}
{Suppose $g(\cdot)$ takes the composite form as expressed in \eqref{eq:g_composite}. Further suppose Assumption A and E hold true.} Let $\{\bx^r\}$ be the sequence generated by the BCM algorithm with G-S rule. Then we have
\begin{align}
\Delta^r=f(\bx^r)-f^*\le \frac{c_7}{\sigma_7}\frac{1}{r}
\end{align}
where
\begin{align}
\sigma_7:&=\frac{\min_{k,j}\eta_{k}^{j}}{2 K I R^2 \max_{k,i}\|A^i_k (A^i_k)^T\|_2(P^i_k)^2 }, \quad
{c}_7:=\max\{4\sigma_7-2, f(\bx^1)-f^*,2\}\nonumber.
\end{align}
\end{corollary}

%\subsubsection{\color{blue}The constrained nonsmooth composite problem}
%
%
%
%Our argument in the previous section also works. That is, first using the composite structure of $\ell^i(\cdot)$'s we should have
%\begin{align}
%f(x^{r})-f(x^{r+1})\ge \sum_{i=1}^{I}\sum_{k=1}^{K}\eta^i_k\|A^i_k(x_k^r-x_k^{r+1})\|^2.
%\end{align}
%
%Then we also have
%\begin{align}
%&\|\nabla_k g(x^{r+1})-\nabla_k g(x_k^{r+1}, w^{r+1}_{-k})\|\nonumber\\
%&=\left\|\sum_{i=1}^{I}(A_k^i)^T \left(\nabla_k\ell^i(A^i_1 x_1^{r+1},\cdots, A^i_K x_K^{r+1})-\nabla_k\ell^i(A^i_1 x_1^{r+1},\cdots, A^i_{k} x_{k}^{r+1}, A^i_{k+1} x^r_{k+1},\cdots, A^i_K x_K^{r})\right)\right\|\nonumber\\
%&\le \sum_{i=1}^{I}\sqrt{\rho\left(A^i_k (A^i_k)^T\right)}P^i_k\sqrt{\sum_{j=1}^{K}\|A^i_j(x_j^{r+1}-x_j^{r})\|^2}.\nonumber
%\end{align}
%Note that in the above notation, $\nabla^i_k\ell(A^i_1 x_1^{r+1},\cdots, A^i_K x_K^{r+1})$ means taking the partial gradient with respect to the $k$th variable of $\ell^i(\cdot)$, which is $A^i_kx_k$. So we can plug this estimate to the cost-to-go estimate derived in the previous section then we are done.

{\begin{remark}
Compared with what we have derived in Theorem \ref{thm:complexity:bcm}, the rate here is explicitly dependent on various problem parameters, hence can be sharpened in certain cases. As an example, consider the simple case where $I=1$ and $\ell(\cdot)=\frac{1}{2}\|\cdot\|^2$. For this problem the Lipschitz continuity constant for the entire smooth part is $M=\|A A^T\|_2$, where $A = [A_1, \cdots, A_K]$. Further we have $P^i_k=\sqrt{K}$, $\eta^i_k=1$ for all $k,i$. When $\|A_k A^T_k\|$'s are approximately the same for all $k$, $\max_{k}\|A_k (A_k)^T\|_2$ is approximately $\frac{1}{K}\|A A^T\|_2$.  This implies that ${\sigma}_7$ is upper bounded by $\frac{1}{2 K R^2 M }$ , which is $K$ times greater than $\sigma_5$ given in \eqref{eq:sigma5}.

\end{remark}}

{
\begin{remark}
Our analysis above implies that when using the BCM (or equivalently the IWFA algorithm \cite{yu04})   %to an even larger family of problem. It is straightforward to see that all we need is that the $k$th subproblem can be {\it transformed} into the following composite form
%\begin{align}\label{eq:g_composite2}
%g(x_k, x_{-k})&=\sum_{i=1}^{I}\ell^i(A^i_{k}x_k, x_{-k})+\sum_{k=1}^{K}b^T_k x_k,
%\end{align}
%where $A^i_{k}$ can be a function of $x_{-k}$, and it satisfies that $0<c_1\le \|A^i_{k} (A^i_{k})^T\|_2\le c_2$ for some positive constants $c_1, c_2$. Therefore our analysis above also applies
to solve the rate optimization problem given in Example \ref{ex:wireless}, a sublinear rate can be obtained regardless of the rank of the channel matrices $\{H_k\}$. To see this, we first check Assumption E-(a).  Denote $X_k:=I_{n_r}+\sum_{j\ne k} H_j C_j H^T_j\succ 0$, then the $k$th subproblem can be reformulated as
\begin{align}
\min_{C_k}-\log\left(|X_k|\left|I_{n_t}+H_k^T X_k^{-1} H_k C_k\right|\right), \quad\st \; C_k\succeq 0, \; \trace[C_k]\le P_k.
\end{align}
Clearly for any feasible choice of $\{C_j\}_{j\ne k}$, we must have
$$|X_k|>0,\quad 0<\left\|H^T_k\left(I+\sum_{j\ne k}P_jH_j H^T_j\right)^{-1} H_k\right\|_2 \le \|H_k^T X_k^{-1} H_k\|_2\le \|H^T_k H_k\|_2.$$
This says that the problem is strongly convex with respect to $H_k^T X_k^{-1} H_k C_k$. It is also easy to verify that the Lipschitz continuous assumption E-(b) is also satisfied; see for example a related discussion in \cite[Section V-A]{scutari13decomposition}. Then Corollary \ref{cor:complexity} implies that IWFA converges in a rate $O(1/r)$, regardless of the rank of the channel matrices. Prior to our work, no convergence rate analysis has been done for the IWFA when solving problem \eqref{eq:MAC}.
\end{remark}}

\subsection{Special Case: The Constrained Nonsmooth L2-SVM Problem}
In this subsection, we assume that $g(x)$ takes the following form (besides Assumption A)
\begin{align}\label{eq:l2-svm}
g(x)=\sum_{i=1}^{I}g_i(x)=\sum_{i=1}^{I}\left[(1-x^T a_i)^+\right]^2=\sum_{i=1}^{I}\left[(1-\sum_{k=1}^{K}x_k^T a_{i,k})^+\right]^2
\end{align}
where $(y)^+$ means $\max\{0,y\}$; $a_{i,k}\in\mathbb{R}^{n_k}$ denotes a subvector of $a_i$ that corresponds to the block $x_k$. This objective is known as the L2 SVM loss. It is easy to observe that the problem is not strongly convex with respect to any block variable $x_k\in\mathbb{R}^{n_k}$. Moreover it is also not a special case of the problems considered in the previous subsection.

To proceed let us define the following short-handed notations:
\begin{align}
%w^{r+1}_{k+1}&:=[x^{r+1}_1,\cdots, x^{r+1}_{k}, x^{r}_{k+1},\cdots, x^{r}_K]\nonumber\\
\ell(y):=\|y\|^2, \quad q_i(x):=(1-x^T a_i)^+ \ge 0. \nonumber
\end{align}
Using these definitions, we have $g_i(x)=\ell(q_i(x))$.

For simplicity, let us consider the BCM scheme with G-S update rule, in which the $k$th block is updated by
\begin{align}
x^{r+1}_k\in \arg\min_{x_k\in X_k} \sum_{i=1}^{I}\ell\left(q_i(x_k,w^{r+1}_{-k})\right)+h_k(x_k).
\end{align}
Moreover, we note that
\begin{align}
\nabla_k g_i(x)=-2a_{i,k} q_i(x),\quad\nabla \ell(y)=2y.
\end{align}
We can obtain the following series of inequalities
\begin{align}
&f(x^r_k, w^{r+1}_{-k})-f(x^{r+1}_k, w^{r+1}_{-k})\nonumber\\
&=f(w^{r+1}_{k})-f(w^{r+1}_{k+1})\nonumber\\
&{\ge} \sum_{i=1}^{I}\nabla _{\color{red}} \ell\left({q}_i(w^{r+1}_{k+1})\right)\left(q_i(w^{r+1}_k)-q_i(w^{r+1}_{k+1}))\right)+h_k(x^r_k)-h_k(x^{r+1}_k)+\frac{1}{2}\|q_i(w^{r+1}_{k+1})-q_i(w^{r+1}_k)\|^2\nonumber\\
&\stackrel{\rm (i)}\ge \sum_{i=1}^{I}{\langle  2q_i(w^{r+1}_{k+1}) \partial q_i(w^{r+1}_{k+1}), x^{r}_k-x^r_{k+1}\rangle}+h_k(x^r_k)-h_k(x^{r+1}_k)+\frac{1}{2}\|q_i(w^{r+1}_{k+1})-q_i(w^{r+1}_k)\|^2\nonumber\\
&\stackrel{\rm (ii)}\ge \frac{1}{2}\sum_{i=1}^{I}\|q_i(w^{r+1}_{k+1})-q_i(w^{r+1}_k)\|^2
\end{align}
where ${\rm (i)}$ is due to the fact that ${\nabla \ell\left(q_i(w^{r+1}_{k+1})\right)=2 q_i(w^{r+1}_{k+1})\ge 0}$ and the fact that $q_i(\cdot)$ is a convex function (albeit nonsmooth); ${\rm (ii)}$ is due to the optimality condition for the $x^{r+1}_{k}$ subproblem. Therefore we have the following sufficient descent estimate
\begin{align}
&f(x^r)-f(x^{r+1})=f(w^{r+1}_1)-f(w^{r+1}_{K+1})\ge \sum_{i=1}^{I}\sum_{j=1}^{K} \frac{1}{2}\|q_i(w^{r+1}_{j+1})-q_i(w^{r+1}_j)\|^2.
\end{align}

Next we proceed to estimate the cost-to-go. To this end, we first
bound $\|\nabla_k g(x^{r+1})-\nabla_{k}g(x^{r+1}_k, w^{r+1}_{-k})\|$. We have the following
\begin{align}
\|\nabla_k g(x^{r+1})-\nabla_{k}g(x^{r+1}_k, w^{r+1}_{-k})\|&=2\left\|\sum_{i=1}^{I}a_{i,k}\left(q_i(x^{r+1})-q_i(w^{r+1}_{k+1})\right)\right\|\nonumber\\
&\le 2\max_{i}\|a_{i,k}\|\sum_{i=1}^{I}\left\|q_i(x^{r+1})-q_i(w^{r+1}_{k+1})\right\|\nonumber\\
&\le 2\max_{i}\|a_{i,k}\|\sum_{i=1}^{I}\sum_{j=1}^{K}\left\|q_i(w_{j}^{r+1})-q_i(w^{r+1}_{j+1})\right\|.\nonumber
\end{align}

Consequently the cost-to-go estimate can be expressed as
\begin{align}
f(x^{r+1})-f(x^*)&\le \langle\nabla g(x^{r+1}), x^{r+1}-x^*\rangle+h(x^{r+1})-h(x^*)\nonumber\\
&=\sum_{k=1}^{K}\langle \nabla_k g(x^{r+1})-\nabla_k g(x^{r+1}_k, w^{r+1}_{-k}), x^{r+1}_k-x^*_k\rangle\nonumber\\
&\quad\quad+\sum_{k=1}^{K}\langle \nabla_k g(x^{r+1}_k, w^{r+1}_{-k}), x^{r+1}_k-x^*_k\rangle+h(x^{r+1})-h(x^*)\nonumber\\
%&\le \sum_{k=1}^{K}\langle \nabla_k g(x^{r+1})-\nabla_k g(x^{r+1}_k, w^{r+1}_{-k}), x^{r+1}_k-x^*_k\rangle\nonumber\\
&\le \sum_{k=1}^{K}\|\nabla_k g(x^{r+1})-\nabla_k g(x^{r+1}_k, w^{r+1}_{-k})\| \|x^{r+1}_k-x^*_k\|\nonumber\\
&\le 2\sum_{k=1}^{K}\max_{i}\|a_{i,k}\|\sum_{i=1}^{I}\sum_{j=1}^{K}\left\|q_i(w_{j}^{r+1})-q_i(w^{r+1}_{j+1})\right\|\|x^{r+1}_k-x^*_k\|.
\end{align}
Finally, we have
\begin{align}
(f(x^{r+1})-f(x^*))^2\le 4\left(\sum_{k=1}^{K}\max_{i}\|a_{i,k}\|\right)^2 K I  R^2\sum_{i=1}^{I}\sum_{j=1}^{K}\|q_i(w_{j}^{r+1})-q_i(w^{r+1}_{j+1})\|^2.
\end{align}

Then we have the following result.
\begin{corollary}\label{cor:complexity3}
Suppose Assumption A holds true, and suppose $g(\cdot)$ takes the form as expressed in \eqref{eq:l2-svm}. Let $\{\bx^r\}$ be the sequence generated by the BCM algorithm with G-S rule. Then we have
\begin{align}
\Delta^r=f(\bx^r)-f^*\le \frac{c_8}{\sigma_8}\frac{1}{r}
\end{align}
where
\begin{align}
\sigma_8:&=\frac{1}{8\left(\sum_{k=1}^{K}\max_{i}\|a_{i,k}\|\right)^2 K I  R^2 }, \quad
{c}_8:=\max\{4\sigma_8-2, f(\bx^1)-f^*,2\}\nonumber.
\end{align}
\end{corollary}

To close this section, we mention that the E-C rule also achieves a sublinear rate for both the composite case and the L2-SVM cases. The analysis follows a similar argument as those presented above, therefore it is not repeated here.

\subsection{Extensions}
We briefly discuss a few extensions of the results presented so far in this section.

The first extension is to the BSUM algorithm without strongly convex upper bounds. For example, to extend Corollary \ref{cor:complexity}, suppose that $g(x)$ is given by \eqref{eq:g_composite}. Further assume that $q_k(y_k; y)$ is an upper bound function for
$$\ell(y_1, \cdots, y_K):=\sum_{i=1}^{I}\ell^i(y_1, y_2,\cdots, y_K)=\sum_{i=1}^{I}\ell^i(A^i_1x_1, A^i_2x_2,\cdots, A^i_K x_K)$$
which is not necessarily strongly convex with respect to $x_k$. If $q_k(y_k;y)$  and $\ell(y_1,\cdots, y_K)$ together satisfy Assumption B for each $k$, then the BSUM algorithm that successively minimizes the upper bounds $q_k$'s achieves a sublinear rate $O(1/r)$.

{Second, our analysis can be directly applied to the algorithm with random permutation of the coordinates between the iterations, a strategy that has been found to be effective in practice \cite[Section 8.5]{shalev14proximaldual}. Indeed, the analysis for both BSUM and BCM  with the G-S rule only requires that within each iteration the coordinates are chosen cyclically. There is no need to maintain the same order across different iterations.}

Third, if the smooth function $g(x)$ is given by the composite form expressed in \eqref{eq:g_composite}, and that the following additional assumptions are satisfied, then the BCM algorithm is capable of linear convergence.
\pn {\bf Assumption F.}
\begin{enumerate}
\item Each $h_k$ satisfies either one of the following
conditions:
\begin{enumerate}
\item The epigraph of $h_k(x_k)$ is a polyhedral set.
\item $h_k(x_k)=\lambda_k\|x_k\|_1+\sum_{J}w_J\|x_{k,J}\|_2$, where
$x_k=(\cdots, x_{k,J},\cdots)$ is a partition of $x_k$ with $J$
being the partition index.
\item Each $h_k(x_k)$ is the sum of the functions described in the previous
two items.
\end{enumerate}
\item  The feasible sets $X_k$, $k=1,\cdots,K$ are polyhedral sets.
\item  Each $A^i_k$ has full column rank.
\end{enumerate}
%For example \cite{Luo92CD, Luo92linear_convergence} derived the linear convergence rate when $h_k\equiv 0$. Reference \cite{tseng09} shows when $h_k(x_k)=\|x_1\|$
The key for proving the linear convergence is to show certain error bound condition holds true for different types of problems. We refer the readers to \cite{Luo92CD, Luo92linear_convergence,tseng09,hong13BSUMM,Sanjabi13} for detailed arguments.

\newsection{Concluding Remarks}
In this paper we have analyzed the iteration complexity of a family of BCD-type algorithms for solving general convex nonsmooth problems of the form \eqref{eq:bcd_problem}. Using a three-step argument, we show that the family of BCD-type algorithms, which includes BCM, BCGD, BCPG algorithms with G-S, E-C, G-So and MBI update rules, converges globally in a sublinear rate of $\mathcal{O}(1/r)$. {It should be noted that in case of the classical BCM algorithm, such sublinear rate can be achieved even without the per-block strong convexity.}  As a future work, it will be interesting to see whether the three-step approach can be extended to establish the iteration complexity bounds for other first order methods.

\bibliographystyle{IEEEbib}

\bibliography{ref,biblio}

\begin{thebibliography}{10}

\bibitem{bertsekas99}
D.~P. Bertsekas,
\newblock {\em Nonlinear Programming, 2nd ed},
\newblock Athena Scientific, Belmont, MA, 1999.

\bibitem{tseng09coordiate}
P.~Tseng and S.~Yun,
\newblock ``A coordinate gradient descent method for nonsmooth separable
  minimization,''
\newblock {\em Mathematical Programming}, vol. 117, pp. 387--423, 2009.

\bibitem{zhang13linear}
H.~Zhang, J.~Jiang, and Z.-Q. Luo,
\newblock ``On the linear convergence of a proximal gradient method for a class
  of nonsmooth convex minimization problems,''
\newblock {\em Journal of the Operations Research Society of China}, vol. 1,
  no. 2, pp. 163--186, 2013.

\bibitem{shalev11}
S.~Shalev-Shwartz and A.~Tewari,
\newblock ``Stochastic methods for $\ell_1$ regularized loss minimization,''
\newblock {\em Journal of Machine Learning Research}, vol. 12, pp. 1865--1892,
  2011.

\bibitem{Beck13}
A.~Beck and L.~Tetruashvili,
\newblock ``On the convergence of block coordinate descent type methods,''
\newblock {\em SIAM Journal on Optimization}, vol. 23, no. 4, pp. 2037--2060,
  2013.

\bibitem{Razaviyayn12SUM}
M.~Razaviyayn, M.~Hong, and Z.-Q. Luo,
\newblock ``A unified convergence analysis of block successive minimization
  methods for nonsmooth optimization,''
\newblock {\em SIAM Journal on Optimization}, vol. 23, no. 2, pp. 1126--1153,
  2013.

\bibitem{tseng09}
P.~Tseng and S.~Yun,
\newblock ``Block-coordinate gradient descent method for linearly constrained
  nonsmooth separable optimization,''
\newblock {\em Journal of Optimization Theory and Applications}, vol. 140, pp.
  513--535, 2009.

\bibitem{nestrov12}
Y.~Nesterov,
\newblock ``Efficiency of coordiate descent methods on huge-scale optimization
  problems,''
\newblock {\em SIAM Journal on Optimization}, vol. 22, no. 2, pp. 341--362,
  2012.

\bibitem{tseng01}
P.~Tseng,
\newblock ``Convergence of a block coordinate descent method for
  nondifferentiable minimization,''
\newblock {\em Journal of Optimization Theory and Applications}, vol. 103, no.
  9, pp. 475--494, 2001.

\bibitem{Chen2012MBI}
B.~Chen, Z.~Li S.~He, and S.~Zhang,
\newblock ``Maximum block improvement and polynomial optimization,''
\newblock {\em SIAM Journal on Optimization}, vol. 22, no. 1, pp. 87--107,
  2012.

\bibitem{Friedman10}
Friedman J, Hastie T, and Tibshirani R.,
\newblock ``Regularization paths for generalized linear models via coordinate
  descent,''
\newblock {\em Journal of Statistical Software}, vol. 33, no. 1, pp. 1--22,
  2010.

\bibitem{Saha10}
A.~Saha and A.~Tewari,
\newblock ``On the nonasymptotic convergence of cyclic coordinate descent
  method,''
\newblock {\em SIAM Journal on Optimization}, vol. 23, no. 1, pp. 576--601,
  2013.

\bibitem{bertsekas96}
D.~P. Bertsekas and J.~N. Tsitsiklis,
\newblock {\em Neuro-Dynamic Programming},
\newblock Athena Scientific, Belmont, MA, 1996.

\bibitem{bertsekas97}
D.~P. Bertsekas and J.~N. Tsitsiklis,
\newblock {\em Parallel and Distributed Computation: Numerical Methods, 2nd
  ed},
\newblock Athena Scientific, Belmont, MA, 1997.

\bibitem{ortega72}
J.~M. Ortega and W.~C. Rheinboldt,
\newblock {\em Iterative Solution of Nonlinear Equations in Several Variables},
\newblock Academic Press, 1972.

\bibitem{Grippo00}
L.~Grippo and M.~Sciandrone,
\newblock ``On the convergence of the block nonlinear {Gauss-Seidel} method
  under convex constraints,''
\newblock {\em Operations Research Letters}, vol. 26, pp. 127--136, 2000.

\bibitem{luo93errorbound:10.1007/BF02096261}
Z.-Q. Luo and P.~Tseng,
\newblock ``Error bounds and convergence analysis of feasible descent methods:
  a general approach,''
\newblock {\em Annals of Operations Research}, vol. 46-47, pp. 157--178, 1993.

\bibitem{Luo92CD}
Z.-Q. Luo and P.~Tseng,
\newblock ``On the convergence of the coordinate descent method for convex
  differentiable minimization,''
\newblock {\em Journal of Optimization Theory and Application}, vol. 72, no. 1,
  pp. 7--35, 1992.

\bibitem{Luo92linear_convergence}
Z.-Q. Luo and P.~Tseng,
\newblock ``On the linear convergence of descent methods for convex essentially
  smooth minimization,''
\newblock {\em SIAM Journal on Control and Optimization}, vol. 30, no. 2, pp.
  408--425, 1992.

\bibitem{Luo93dual}
Z.-Q. Luo and P.~Tseng,
\newblock ``On the convergence rate of dual ascent methods for strictly convex
  minimization.,''
\newblock {\em Mathematics of Operations Research}, vol. 18, no. 4, pp.
  846--867, 1993.

\bibitem{tseng09approximation}
P.~Tseng,
\newblock ``Approximation accuracy, gradient methods, and error bound for
  structured convex optimization,''
\newblock {\em Mathematical Programming}, vol. 125, no. 2, pp. 263--295, 2010.

\bibitem{richtarik12}
P.~Richtarik and M.~Takac,
\newblock ``Iteration complexity of randomized block-coordinate descent methods
  for minimizing a composite function,''
\newblock {\em Mathematical Programming}, vol. 144, pp. 1--38, 2014.

\bibitem{lu13complexity}
Z.~Lu and L.~Xiao,
\newblock ``On the complexity analysis of randomized block-coordinate descent
  methods,''
\newblock 2013,
\newblock accepted by Mathematical Programming.

\bibitem{Beck13b}
A.~Beck,
\newblock ``On the convergence of alternating minimization with applications to
  iteratively reweighted least squares and decomposition schemes,''
\newblock {\em SIAM Journal on Optimization}, vol. 25, no. 1, pp. 185--209,
  2015.

\bibitem{HeLiaoHanYang2002}
B.~He, L.~Liao, D.~Han, and H.~Yang,
\newblock ``A new inexact alternating directions method for monotone
  variational inequalities,''
\newblock {\em Mathematical Programming}, vol. 92, no. 1, pp. 103--118, 2002.

\bibitem{WangYuan2012}
X.~Wang and X.~Yuan,
\newblock ``The linearized alternating direction method of multipliers for
  dantzig selector,''
\newblock {\em SIAM Journal on Scientific Computing}, vol. 34, no. 5, pp.
  2792--2811, 2012.

\bibitem{zhang11primaldual}
X.~Zhang, M.~Burger, and S.~Osher,
\newblock ``A unified primal-dual algorithm framework based on {B}regman
  iteration,''
\newblock {\em Journal of Scientific Computing}, vol. 46, no. 1, pp. 20--46,
  2011.

\bibitem{Yang10ADMM}
J.~Yang, Y.~Zhang, and W.~Yin,
\newblock ``A fast alternating direction method for {TVL1-L2} signal
  reconstruction from partial fourier data,''
\newblock {\em IEEE Journal of Selected Topics in Signal Processing}, vol. 4,
  no. 2, pp. 288--297, 2010.

\bibitem{hong15busmm_spm}
M.~Hong, M.~Razaviyayn, Z.-Q. Luo, and J.-S. Pang,
\newblock ``A unified algorithmic framework for block-structured optimization
  involving big data,''
\newblock 2015,
\newblock submitted for publication.

\bibitem{Combettes09}
P.~Combettes and J.-C. Pesquet,
\newblock ``Proximal splitting methods in signal processing,''
\newblock in {\em Fixed-Point Algorithms for Inverse Problems in Science and
  Engineering}, Springer Optimization and Its Applications, pp. 185--212.
  Springer New York, 2011.

\bibitem{Mairal13}
J.~Mairal,
\newblock ``Optimization with first-order surrogate functions,''
\newblock in {\em International Conference on Machine Learning (ICML).}, 2013.

\bibitem{cover05}
T.~M. Cover and J.~A. Thomas,
\newblock {\em Elements of Information Theory, second edition},
\newblock Wiley, 2005.

\bibitem{yu04}
W.~Yu, W.~Rhee, S.~Boyd, and J.~M. Cioffi,
\newblock ``Iterative water-filling for {G}aussian vector multiple-access
  channels,''
\newblock {\em IEEE Transactions on Information Theory}, vol. 50, no. 1, pp.
  145--152, 2004.

\bibitem{Nesterov04}
Y.~Nesterov,
\newblock {\em Introductory lectures on convex optimization: A basic course},
\newblock Springer, 2004.

\bibitem{Daubechies10}
I.~Daubechies, R.~DeVore, M.~Fornasier, and C.~S. Gunturk,
\newblock ``Iteratively reweighted least squares minimization for sparse
  recovery,''
\newblock {\em Communications on Pure and Applied Mathematics}, vol. 63, no. 1,
  pp. 1--38, 2010.

\bibitem{Hiriart-Urruty}
J.-B. Hiriart-Urruty and C.~Lemarechal,
\newblock {\em Convex Analysis and Minimization Algorithms I: Fundamentals},
\newblock Springer, 1996.

\bibitem{Beck:2009:FIS:1658360.1658364}
A.~Beck and M.~Teboulle,
\newblock ``A fast iterative shrinkage-thresholding algorithm for linear
  inverse problems,''
\newblock {\em SIAM Journal on Imgaging Science}, vol. 2, no. 1, pp. 183--202,
  2009.

\bibitem{tseng08acc}
P.~Tseng,
\newblock ``On accelerated proximal gradient methods for convex-concave
  optimization,''
\newblock 2008,
\newblock preprint.

\bibitem{scutari13decomposition}
G.~Scutari, F.~Facchinei, P.~Song, D.~P. Palomar, and J.-S. Pang,
\newblock ``Decomposition by partial linearization: Parallel optimization of
  multi-agent systems,''
\newblock {\em IEEE Transactions on Signal Processing}, vol. 63, no. 3, pp.
  641--656, 2014.

\bibitem{shalev14proximaldual}
S.~Shalev-Shwartz and T.~Zhang,
\newblock ``Proximal stochastic dual coordinate ascent methods for regularzied
  loss minimization,''
\newblock {\em Journal of Machine Learning Rsearch}, vol. 14, pp. 567--599,
  2013.

\bibitem{hong13BSUMM}
M.~Hong, T.-H. Chang, X.~Wang, M.~Razaviyayn, S.~Ma, and Z.-Q. Luo,
\newblock ``A block successive upper bound minimization method of multipliers
  for linearly constrained convex optimization,''
\newblock 2013,
\newblock Preprint, available online arXiv:1401.7079.

\bibitem{Sanjabi13}
M.~Sanjabi, M.~Kadkhodaei, and Z.-Q. Luo,
\newblock ``On the linear convergence of approximate proximal splitting methods
  for non-smooth convex minimization,''
\newblock 2012,
\newblock manuscript.

\end{thebibliography}

\end{document}